\documentclass[11pt, reqno]{amsart}
\usepackage{amsfonts}
\usepackage{amsmath}
\usepackage{amssymb, color}
\usepackage{amscd}
\usepackage[mathscr]{eucal}
\usepackage{url}

\allowdisplaybreaks[1]

\newcommand{\ph}[2]{{\left({#1}\right)}_{#2}}

\newcommand{\gf}[1]{\Gamma{\left({#1}\right)}}

\renewcommand*{\bar}{\overline}
\newcommand{\gfp}[1]{\Gamma_p{\left({#1}\right)}}
\newcommand{\biggfp}[1]{\Gamma_p{\bigl({#1}\bigr)}}

\makeatletter
\def\imod#1{\allowbreak\mkern2mu({\operator@font mod}\,\,#1)}
\makeatother

\makeatletter
\def\jmod#1{\allowbreak\mkern5mu({\operator@font mod}\,\,#1)}
\makeatother

\def\I{\mathrm{i}}

\theoremstyle{plain}

\newtheorem{theorem}{Theorem}[section]
\newtheorem{lemma}[theorem]{Lemma}
\newtheorem{prop}[theorem]{Proposition}
\newtheorem{cor}[theorem]{Corollary}
\theoremstyle{definition}
\newtheorem{defi}[theorem]{Definition}

\numberwithin{equation}{section}

\makeatletter
\def\imod#1{\allowbreak\mkern5mu({\operator@font mod}\,\,#1)}
\makeatother

\begin{document}

\title[Splitting Hypergeometric Functions]{Splitting Hypergeometric Functions\\over Roots of Unity}

\author{Dermot M\lowercase{c}Carthy}
\address{Dermot M\lowercase{c}Carthy, Department of Mathematics \& Statistics\\
Texas Tech University\\
Lubbock, TX 79410-1042\\
USA}
\email{dermot.mccarthy@ttu.edu}

\author{Mohit Tripathi}
\address{Mohit Tripathi, Department of Mathematics \& Statistics\\
Texas Tech University\\
Lubbock, TX 79410-1042\\
USA}
\email{mohit.tripathi@ttu.edu}


\begin{abstract}
We examine hypergeometric functions in the finite field, $p$-adic and classical settings.
In each setting, we prove a formula which splits the hypergeometric function into a sum of lower order functions whose arguments differ by roots of unity.
We provide multiple applications of these results, including new reduction and summation formulas for finite field hypergeometric functions, along with classical analogues; evaluations of special values of these functions which apply in both the finite field and $p$-adic settings; and new relations to Fourier coefficients of modular forms. 
\end{abstract}

\maketitle

%%%%%%%%%%%%%%%%%%%%%%%%%%%%%%%%%%%%%%%%%%%%%%%%%%
%%%%%%%%%%%%%%%%%%%%%%%%%%%%%%%%%%%%%%%%%%%%%%%%%%
%%%%%%%%%%%%%%%%%%%%%%%%%%%%%%%%%%%%%%%%%%%%%%%%%%
%%%%%%%%%%%%%%%%%%%%%%%%%%%%%%%%%%%%%%%%%%%%%%%%%%
%%%%%%%%%%%%%%%%%%%%%%%%%%%%%%%%%%%%%%%%%%%%%%%%%%
%%%%%%%%%%%%%%%%%%%%%%%%%%%%%%%%%%%%%%%%%%%%%%%%%%

\section{Introduction}\label{sec_Intro}
Finite field hypergeometric functions were originally defined by Greene \cite{G, G2} as analogues of classical hypergeometric series.
Functions of this type were also introduced by Katz \cite{Ka} about the same time.
Greene's work includes an extensive catalogue of transformation and summation formulas for these functions, mirroring those in the classical case.
These functions also have a nice character sum representation and so the transformation and summation formulas can be interpreted as relations to simplify and evaluate complex character sums \cite{EG, GS}.
Using character sums to count the number of solutions to equations over finite fields is a long established practice \cite{DH, HL, Ko4, W} and finite field hypergeometric functions also naturally lend themselves to this endeavor \cite{AO, BKS, F2, FLRST, Go, K, L2, McC12, O}.
Following the modularity theorem, and, by then, known links between finite field hypergeometric functions and elliptic curves, many authors began examining links between these functions and Fourier coefficients of modular forms \cite{DM, DM2, E, FOP, FM, Ki, LTYZ, McC5, MP, M, P, RV}. More recently, finite field hypergeometric functions have played a central role in the theory of hypergeometric motives, which has led to increased interest in the functions and their properties \cite{AGLT, BCM, DKSSVW, RoRV}.

While hypergeometric functions over finite fields were originally defined by Greene, we will use a normalized version defined by the first author \cite{McC6, McC12}. Throughout, let $q=p^r$ be a prime power.  Let $\mathbb{F}_q$ be the finite field with $q$ elements, and let $\widehat{\mathbb{F}_q^*}$ be the group of multiplicative characters of $\mathbb{F}_q^*$.  We extend the domain of $\chi \in \widehat{\mathbb{F}^{*}_{q}}$ to $\mathbb{F}_{q}$ by defining $\chi(0):=0$ (including for the trivial character $\varepsilon$) and denote $\bar{\chi}$ as the inverse of $\chi$. We denote by $\varphi$ the character of order two in $\widehat{\mathbb{F}_q^*}$, when $q$ is odd. 
More generally, for ${k>2}$ a positive integer, we let $\chi_k\in\widehat{\mathbb{F}_q^*}$ denote a character of order $k$ when $q \equiv 1 \imod{k}$. 
Let $\theta$ be a fixed non-trivial additive character of $\mathbb{F}_q$, and for $\chi \in \widehat{\mathbb{F}^{*}_{q}}$ define the Gauss sum $g(\chi):= \sum_{x \in \mathbb{F}_q} \chi(x) \theta(x)$.
We define the finite field hypergeometric function as follows. 
\begin{defi}[\cite{McC6, McC12}]\label{def_F} 
For $A_1, A_2, \dotsc, A_m, B_1, B_2 \dotsc, B_m \in \widehat{\mathbb{F}_q^{*}}$ and $\lambda \in \mathbb{F}_q$,
\begin{equation*}\label{def_HypFnFF}
{_{m}F_{m}} {\biggl( \begin{array}{cccc} 
A_1, & A_2, & \dotsc, & A_m \\
B_1, & B_2, & \dotsc, & B_m \end{array}
\Big| \; \lambda \biggr)}_{q}\\
:= \frac{-1}{q-1}  \sum_{\chi \in \widehat{\mathbb{F}_q^{*}}}
\prod_{i=1}^{m} \frac{g(A_i \chi)}{g(A_i)} \frac{g(\bar{B_i \chi})}{g(\bar{B_i})}
 \chi(-1)^{m}
 \chi(\lambda).
\end{equation*}
\end{defi}

In recent years, the second author, with Barman and others, \cite{Mohit-1,TB-RNT-2018, TB-RNT-2020, TB-RMS-2021, BST} have developed the theory of the finite field Appell functions. They establish several product and summation formulas for these functions, which has also led to new relations for ${_{m}F_{m}}$, as defined above. 
One such relation can be derived from recent work of the second author and Meher \cite{TM}.
\begin{theorem}\cite{TM}\label{thm_Mohit}
For $A, C \in \widehat{\mathbb{F}_q^{*}}$ such that $A^2 \neq \varepsilon$ and $C^2 = \varepsilon$,
\begin{equation*}
{_{2}F_{2}} {\biggl( \begin{array}{cc} 
A^2, & \bar{A}^2 \\[1pt]
\varepsilon, & \varepsilon \end{array}
\Big| \; \lambda \biggr)}_{q}
+
{_{2}F_{2}} {\biggl( \begin{array}{cc} 
A^2, & \bar{A}^2 \\[1pt]
\varepsilon, & \varepsilon \end{array}
\Big| -\lambda \biggr)}_{q}
=
{_{4}F_{4}} {\biggl( \begin{array}{cccc} 
A, & \varphi A, & \bar{A}, & \varphi \bar{A}\\[1pt]
\varepsilon, & \varphi, & C, & \varphi C \end{array}
\Big| \; \lambda^2 \biggr)}_{q}.
\end{equation*}
\end{theorem}
\noindent
The main purpose of this paper is to generalize Theorem \ref{thm_Mohit} in multiple directions, using more direct methods. 
Specifically, we generalize Theorem \ref{thm_Mohit} from $_2F_2$ to $_mF_m$ for any $m$, allowing any character arguments, without restrictions, where the sum is over all roots of unity, not just $\pm1$, times $\lambda$.
We then extend these results in the $p$-adic setting.
We also prove a classical hypergeometric series analogue.
All these results will be outlined in Section \ref{sec_Results}.
In Section \ref{sec_Apps}, we will provide multiple applications of these generalized results, including new reduction and summation formulas for finite field hypergeometric functions, along with classical analogues; evaluations of special values of these functions which apply in both the finite field and $p$-adic settings; and new relations to Fourier coefficients of modular forms. 
After outlining some preliminaries in Section \ref{sec_Prelim}, we will prove our main results in Section \ref{sec_Proofs}.
Section \ref{sec_GF} contains a comprehensive discussion on the relationship between our finite field and $p$-adic hypergeometric functions, and, in particular, how this relationship is affected if the parameters, or some subset thereof, are ``\emph{defined over $\mathbb{Q}$}''. 

%%%%%%%%%%%%%%%%%%%%%%%%%%%%%%%%%%%%%%%%%%%%%%%%%%
%%%%%%%%%%%%%%%%%%%%%%%%%%%%%%%%%%%%%%%%%%%%%%%%%%
%%%%%%%%%%%%%%%%%%%%%%%%%%%%%%%%%%%%%%%%%%%%%%%%%%
%%%%%%%%%%%%%%%%%%%%%%%%%%%%%%%%%%%%%%%%%%%%%%%%%%
%%%%%%%%%%%%%%%%%%%%%%%%%%%%%%%%%%%%%%%%%%%%%%%%%%
%%%%%%%%%%%%%%%%%%%%%%%%%%%%%%%%%%%%%%%%%%%%%%%%%%

\section{Statement of Main Results}\label{sec_Results}
\subsection{Finite Field Setting}\label{subsec_Res_FF}
Our first result generalizes Theorem \ref{thm_Mohit} from $_2F_2$ to $_mF_m$ for any $m$, allowing any character arguments, without restrictions, where the sum is over all roots of unity, not just $\pm1$, times $\lambda$. 
\begin{theorem}\label{thm_F_Main}
Let $n$ be a positive integer and let $q \equiv 1 \imod n$ be a prime power. Let $\zeta_n$ be a primitive $n$-th root of unity in $\mathbb{F}_q^{*}$ and let $\chi_n \in \widehat{\mathbb{F}_q^{*}}$ denote a character of order $n$. Then for $A_1, A_2, \dotsc, A_m, B_1, B_2 \dotsc, B_m \in \widehat{\mathbb{F}_q^{*}}$,
\begin{multline*}
\sum_{l=0}^{n-1}
{_{m}F_{m}} {\biggl( \begin{array}{cccc} A_1^n, & A_2^n, & \dotsc, & A_m^n \\[1pt]
 B_1^n, & B_2^n, & \dotsc, & B_m^n \end{array}
\Big| \; \zeta_n^l \cdot \lambda \biggr)}_{q}\\
=
{_{nm}F_{nm}} {\biggl( \begin{array}{c} 
A_i \,\chi_n^l \, : 1 \leq i \leq m, 0 \leq l \leq n-1  \\[1pt]
B_i \, \chi_n^l  \, : 1 \leq i \leq m, 0 \leq l \leq n-1 \end{array}
\Big| \; \lambda^n \biggr)}_{q}.
\end{multline*}
\end{theorem}
\noindent
We also have the following converse result.
\begin{theorem}\label{thm_F_Main_Converse}
Let $n$ be a positive integer and let $q \equiv 1 \imod n$ be a prime power. Let $\chi_n \in \widehat{\mathbb{F}_q^{*}}$ denote a character of order $n$. Let $A_1, A_2, \dotsc, A_m, B_1, B_2 \dotsc, B_m \in \widehat{\mathbb{F}_q^{*}}$. If $\lambda \in \mathbb{F}_q$ is not an $n$-th power, then
\begin{equation*}
{_{nm}F_{nm}} {\biggl( \begin{array}{c} 
A_i \,\chi_n^l \, : 1 \leq i \leq m, 0 \leq l \leq n-1  \\[1pt]
B_i \, \chi_n^l  \, : 1 \leq i \leq m, 0 \leq l \leq n-1 \end{array}
\Big| \; \lambda \biggr)}_{q}=
0.
\end{equation*}
\end{theorem}
\noindent
When $n=2$, the following corollaries are immediate.
\begin{cor}\label{cor_F_n=2}
For $q$ an odd prime power,
\begin{multline*}
{_{m}F_{m}} {\biggl( \begin{array}{cccc} 
A_1^2, & A_2^2, & \dotsc, & A_m^2 \\[1pt]
B_1^2, & B_2^2, & \dotsc, & B_m^2 \end{array}
\Big| \; \lambda \biggr)}_{q}
+
{_{m}F_{m}} {\biggl( \begin{array}{cccc}
A_1^2, & A_2^2, & \dotsc, & A_m^2 \\[1pt]
B_1^2, & B_2^2, & \dotsc, & B_m^2 \end{array}
\Big|  -\lambda \biggr)}_{q}\\
=
{_{2m}F_{2m}} {\biggl( \begin{array}{ccccccc} 
A_1, & \varphi A_1, & A_2, & \varphi A_2, & \dotsc, & A_m, & \varphi A_m\\[1pt]
B_1, & \varphi B_1, & B_2, & \varphi B_2, & \dotsc, & B_m, & \varphi B_m \end{array}
\Big| \; \lambda^2 \biggr)}_{q}.
\end{multline*}
\end{cor}
\noindent
Taking $m=2$, $A_1=A, A_2=\bar{A}, B_1=B_2=\varepsilon$ in Corollary \ref{cor_F_n=2} recovers Theorem \ref{thm_Mohit}.
\begin{cor}\label{cor_F_n=2_Converse}
Let $q$ be an odd prime power. If $\lambda \in \mathbb{F}_q$ is not a square, then
\begin{equation*}
{_{2m}F_{2m}} {\biggl( \begin{array}{ccccccc} 
A_1, & \varphi A_1, & A_2, & \varphi A_2, & \dotsc, & A_m, & \varphi A_m\\[1pt]
B_1, & \varphi B_1, & B_2, & \varphi B_2, & \dotsc, & B_m, & \varphi B_m \end{array}
\Big| \;\lambda \biggr)}_{q}
=0.
\end{equation*}
\end{cor}

Let $\mathbb{F}_q^{*} = \langle T \rangle$ and $A_i, B_i \in \widehat{\mathbb{F}_q^{*}}$. 
Then $A_i = T^{a_i (q-1)}$ and $B_i = T^{b_i(q-1)}$ for some $a_i, b_i \in \mathbb{Q}$, such that $a_i (q-1), b_i (q-1) \in \mathbb{Z}$. 
Many applications of $_mF_m(\{A_i\}; \{B_i\} \mid \lambda)_q$ require fixed $a_i, b_i$. 
For fixed $a_i, b_i$, if we consider $_mF_m(\{A_i\}; \{B_i\} \mid \lambda)_q$ to be a function of $q$, then the domain of this function is all $q \equiv 1 \imod d$, where $d$ is the least common denominator of of the elements in $\{a_i\} \cup \{b_i\}$. The first author has defined a function in the $p$-adic setting which extends $_mF_m(\{A_i\}; \{B_i\} \mid \lambda)_q$, and whose domain is all $q$ relatively prime to $d$. We use this function to extend the domain of all results stated above.

%%%%%%%%%%%%%%%%%%%%%%%%%%%%%%%%%%%%%%%%%%%%%%%%%%
%%%%%%%%%%%%%%%%%%%%%%%%%%%%%%%%%%%%%%%%%%%%%%%%%%
%%%%%%%%%%%%%%%%%%%%%%%%%%%%%%%%%%%%%%%%%%%%%%%%%%

\subsection{$p$-adic Setting}\label{subsec_Res_padic}

Let $\mathbb{Z}_p$ denote the ring of $p$-adic integers and let $\gfp{\cdot}$ denote Morita's $p$-adic gamma function. Let $\omega$ denote the Teichm\"{u}ller character of $\mathbb{F}_q$ (see Section \ref{sec_Prelim} for full details) with $\bar{\omega}$ denoting its character inverse. 
For $x \in \mathbb{Q}$, we let  $\left\lfloor x \right\rfloor$ denote the greatest integer less than or equal to $x$ and $\langle x \rangle$ denote the fractional part of $x$, i.e. $\langle x \rangle = x- \left\lfloor x \right\rfloor$.

\begin{defi}\cite[Definition 5.1]{McC7}\label{def_Gq}
Let $q=p^r$ for $p$ an odd prime. Let $\lambda \in \mathbb{F}_q$, $m \in \mathbb{Z}^{+}$ and $a_i, b_i \in \mathbb{Q} \cap \mathbb{Z}_p$, for $1 \leq i \leq m$.
Then define  
\begin{multline*}
{_{m}G_{m}}
\biggl[ \begin{array}{cccc} 
a_1, & a_2, & \dotsc, & a_m \\[2pt]
b_1, & b_2, & \dotsc, & b_m \end{array}
\Big| \; \lambda \; \biggr]_q
: = \frac{-1}{q-1}  \sum_{j=0}^{q-2} 
(-1)^{jm}\;
\bar{\omega}^j(\lambda)\\
\times 
\prod_{i=1}^{m} 
\prod_{k=0}^{r-1} 
\frac{\biggfp{\langle (a_i -\frac{j}{q-1} )p^k \rangle}}{\biggfp{\langle a_i p^k \rangle}}
\frac{\biggfp{\langle (-b_i +\frac{j}{q-1}) p^k \rangle}}{\biggfp{\langle -b_i p^k\rangle}}
(-p)^{-\lfloor{\langle a_i p^k \rangle -\frac{j p^k}{q-1}}\rfloor -\lfloor{\langle -b_i p^k\rangle +\frac{j p^k}{q-1}}\rfloor}.
\end{multline*}
\end{defi}
\noindent
We note that the value of ${_{m}G_{m}}[\cdots]$ depends only on the fractional part of the $a_i$ and $b_i$ parameters, and is invariant if we change the order of the parameters. For fixed  parameters $\{a_i\}$ and $\{b_i\}$, ${_{m}G_{m}}[\cdots]_q$ is defined at any odd prime power $q=p^r$ where $p$ is relatively prime to the denominators of the $a_i$'s and $b_i$'s, i.e., all $a_i, b_i \in \mathbb{Z}_p$.

${_{m}G_{m}}[\cdots]_q$ extends $_mF_m(\cdots)_q$ via the following relation.
\begin{lemma}[c.f. \cite{McC7} Lemma 3.3]\label{lem_G_to_F}
For an odd prime power $q$, let $A_i, B_i \in \widehat{\mathbb{F}_q^{*}}$ be given by $\bar{\omega}^{a_i(q-1)}$ and $\bar{\omega}^{b_i(q-1)}$ respectively, where $\omega$ is the Teichm\"{u}ller character of $\mathbb{F}_q$. Then
\begin{equation*}
{_{m}F_{m}} {\biggl( \begin{array}{cccc} A_1, & A_2, & \dotsc, & A_m \\
 B_1, & B_2, & \dotsc, & B_m \end{array}
\Big| \; \lambda \biggr)}_{q}
=
{_{m}G_{m}}
\biggl[ \begin{array}{cccc} a_1, & a_2, & \dotsc, & a_m \\
 b_1, & b_2, & \dotsc, & b_m \end{array}
\Big| \; \lambda^{-1} \; \biggr]_q.
\end{equation*}
\end{lemma}

Following \cite{BCM}, we say the parameters $(\{a_i\}, \{b_i\})$ are \emph{defined over $\mathbb{Q}$} if there exist positive integers $p_1, p_2, \dots, p_t$ and $q_1, q_2, \dots, q_s$, with $t$ and $s$ minimal, such that
$$\prod_{i=1}^{m} \frac{x- e^{2 \pi \I a_i}}{x- e^{2 \pi \I b_i}} = \frac{\prod_{i=1}^{t} x^{p_i} - 1}{\prod_{i=1}^{s} x^{q_i} - 1}.$$
Let $D(x) := \gcd({\prod_{i=1}^{t} x^{p_i} - 1},{\prod_{i=1}^{s} x^{q_i} - 1})$, of degree $\delta$, have zeros $\exp \big(\frac{2 \pi \I \, c_h}{q-1}\bigr)$ for $1 \leq h \leq \delta$. 
Let $\mathfrak{s}(c)$ denote the multiplicity of the zero $\exp \bigl(\frac{2 \pi \I \, c}{q-1}\bigr)$.
We note that $m+\delta = \sum_{i=1}^{t} p_i = \sum_{i=1}^{s}q_i$ and we define 
$$M:=\frac{\prod_{i=1}^{t}{p_i}^{p_i}}{\prod_{i=1}^{s}{q_i}^{q_i}}.$$

\begin{theorem}\label{thm_GQ}
Let $q=p^r$ for $p$ an odd prime. If $(\{a_i\}, \{b_i\})$ are defined over $\mathbb{Q}$ with corresponding exponents $(\{p_i : 1 \leq i \leq t \}, \{q_i  : 1 \leq i \leq s \})$, then
\begin{multline*}
{_{m}G_{m}}
\biggl[ \begin{array}{cccc} a_1, & a_2, & \dotsc, & a_m \\[2pt]
 b_1, & b_2, & \dotsc, & b_m \end{array}
\Big| \; \lambda \; \biggr]_q
= 
\frac{-1}{q-1}  \sum_{j=0}^{q-2} 
(-1)^{j(m+\delta)}\;
q^{-\mathfrak{s}(0)+\mathfrak{s}(j)}\;
\bar{\omega}^j(M \cdot \lambda)\\
\times
\prod_{k=0}^{r-1} 
\prod_{i=1}^{t}  
\biggfp{\langle \tfrac{-j p_i}{q-1} p^k \rangle}
%(-p)^{\langle \tfrac{-j p_i}{q-1} p^k \rangle}
(-p)^{-\lfloor{ \tfrac{-j p_i}{q-1} p^k} \rfloor}
\prod_{i=1}^{s}  
\biggfp{\langle \tfrac{j q_i}{q-1} p^k \rangle}
%(-p)^{\langle \tfrac{j q_i}{q-1} p^k \rangle}.
(-p)^{-\lfloor{ \tfrac{j q_i}{q-1} p^k} \rfloor}.
\end{multline*}
\end{theorem}

We now extend the results of Section \ref{subsec_Res_FF}. 
\begin{theorem}\label{thm_G_Main}
Let $n$ be a positive integer and let $q \equiv 1 \imod n$ be an odd prime power. Let $\zeta_n$ be a primitive $n$-th root of unity in $\mathbb{F}_q^{*}$ .
%Let $\lambda \in \mathbb{F}_q$, $m \in \mathbb{Z}^{+}$ and, for $1 \leq i \leq m$, $a_i, b_i \in \mathbb{Q} \cap \mathbb{Z}_p$.
If $(\{na_i\}, \{nb_i\})$ are defined over $\mathbb{Q}$, then
\begin{multline*}
\sum_{l=0}^{n-1}
{_{m}G_{m}}
\biggl[ \begin{array}{cccc} 
n a_1, & n a_2, & \dotsc, & n a_m \\[2pt]
n b_1, & n b_2, & \dotsc, & n b_m \end{array}
\Big| \; \zeta_n^l \cdot \lambda \; \biggr]_q\\
=
{_{nm}G_{nm}}
\biggl[ \begin{array}{c} 
a_i+\frac{l}{n} : 1 \leq i \leq m, 0 \leq l \leq n-1 \\[2pt]
b_i +\frac{l}{n} : 1 \leq i \leq m, 0 \leq l \leq n-1 \end{array}
\Big| \; \lambda^n \; \biggr]_q.
\end{multline*}
Furthermore, if $(\{na_i\}, \{nb_i\})$ are defined over $\mathbb{Q}$ with corresponding exponents $(\{p_i : 1 \leq i \leq t \}, \{q_i  : 1 \leq i \leq s \})$ then 
$(\{a_i+\frac{l}{n}\}, \{b_i+\frac{l}{n}\})$ are defined over $\mathbb{Q}$ with exponents $(\{n p_i  : 1 \leq i \leq t \}, \{n q_i : 1 \leq i \leq s\})$.
\end{theorem}
\noindent
We also have the following converse result, regardless of whether the parameters are defined over $\mathbb{Q}$.
\begin{theorem}\label{thm_G_Main_Converse}
Let $n$ be a positive integer and let $q \equiv 1 \imod n$ be a prime power. 
If $\lambda \in \mathbb{F}_q$ is not an $n$-th power, then
\begin{equation*}
{_{nm}G_{nm}}
\biggl[ \begin{array}{c} 
a_i+\frac{l}{n} : 1 \leq i \leq m, 0 \leq l \leq n-1 \\[2pt]
b_i+\frac{l}{n} : 1 \leq i \leq m, 0 \leq l \leq n-1 \end{array}
\Big| \; \lambda \; \biggr]_q
=0.
\end{equation*}
\end{theorem}

\noindent
When $n=2$, the following corollaries are immediate.
\begin{cor}\label{cor_G_n=2}
For $q$ an odd prime power and $(\{2a_i\}, \{2b_i\})$ defined over $\mathbb{Q}$,
\begin{multline*}
{_{m}G_{m}}
\biggl[ \begin{array}{cccc} 2 a_1, & 2 a_2, & \dotsc, & 2 a_m \\[2pt]
 2 b_1, & 2 b_2, & \dotsc, & 2 b_m \end{array}
\Big| \; \lambda \; \biggr]_q
+
{_{m}G_{m}}
\biggl[ \begin{array}{cccc} 2 a_1, & 2 a_2, & \dotsc, & 2 a_m \\[2pt]
 2 b_1, & 2 b_2, & \dotsc, & 2 b_m \end{array}
\Big| \; -\lambda \; \biggr]_q\\
=
{_{2m}G_{2m}}
\biggl[ \begin{array}{ccccccc} 
a_1, & a_1 +\frac{1}{2}, & a_2, & a_2 +\frac{1}{2}, & \dotsm, & a_m, & a_m +\frac{1}{2} \\[2pt]
b_1, & b_1 +\frac{1}{2}, & b_2, & b_2 +\frac{1}{2}, & \dotsm, & b_m, & b_m +\frac{1}{2}
 \end{array}
\Big| \; \lambda^2 \; \biggr]_q.
\end{multline*}
\end{cor}

\begin{cor}\label{cor_G_n=2_Converse}
Let $q$ be an odd prime power. If $\lambda \in \mathbb{F}_q$ is not a square, then
\begin{equation*}
{_{2m}G_{2m}}
\biggl[ \begin{array}{ccccccc} 
a_1, & a_1 +\frac{1}{2}, & a_2, & a_2 +\frac{1}{2}, & \dotsm, & a_m, & a_m +\frac{1}{2} \\[2pt]
b_1, & b_1 +\frac{1}{2}, & b_2, & b_2 +\frac{1}{2}, & \dotsm, & b_m, & b_m +\frac{1}{2}
 \end{array}
\Big| \; \lambda \; \biggr]_q
=0.
\end{equation*}
\end{cor}

%%%%%%%%%%%%%%%%%%%%%%%%%%%%%%%%%%%%%%%%%%%%%%%%%%
%%%%%%%%%%%%%%%%%%%%%%%%%%%%%%%%%%%%%%%%%%%%%%%%%%
%%%%%%%%%%%%%%%%%%%%%%%%%%%%%%%%%%%%%%%%%%%%%%%%%%

\subsection{Classical Setting}\label{subsec_Res_Classical}
Recall the classical generalized hypergeometric series $_mF_{m}$ defined by
\begin{equation*}
{_mF_{m}} 
\biggl[ \begin{array}{cccc} 
a_1, & a_2, & \dotsc, & a_m \vspace{.05in}\\
b_1, & b_2, & \dotsc, & b_m \end{array}
\Big| \; z \biggr]
:=\sum^{\infty}_{k=0}
\frac{\ph{a_1}{k} \ph{a_2}{k} \dotsm \ph{a_m}{k}}
{\ph{b_1}{k} \ph{b_2}{k} \dotsm \ph{b_m}{k}}
\; {z^k},
\end{equation*}
where $a_i$, $b_i$ and $z$ are complex numbers, with none of the $b_i$ being negative integers or zero, $m$ a positive integer, $\ph{a}{0}:=1$ and $\ph{a}{k} := a(a+1)(a+2)\dotsm(a+k-1)$ for a positive integer $k$. Setting $b_1=1$ recovers the more usual ${_mF_{m-1}}$ definition/notation. 

We have the following classical series analogue of Theorems \ref{thm_F_Main} and \ref{thm_G_Main}.
\begin{theorem}\label{thm_C_Main}
Let $n$ be a positive integer and let $\xi_n \in \mathbb{C}$ be a primitive $n$-th root of unity. 
Then,
\begin{multline*}
\sum_{l=0}^{n-1}
{_mF_{m}} 
\biggl[ \begin{array}{cccc} 
n a_1, & n a_2, & \dotsc, & n a_m \vspace{.05in}\\
n b_1 & n b_2, & \dotsc, & n b_m \end{array}
\Big| \; \xi_n^l \cdot z \biggr]\\
=
n \cdot
{_{nm}F_{nm}} 
{\biggl[ \begin{array}{c} 
a_i+\frac{l}{n} : 1 \leq i \leq m, 0 \leq l \leq n-1 \\[2pt]
b_i +\frac{l}{n} : 1 \leq i \leq m, 0 \leq l \leq n-1 \end{array}
\Big| \; z^n \biggr]}.
\end{multline*}
\end{theorem}
Theorem \ref{thm_C_Main} can be derived from \cite[eqn.~24, p.~440]{PBM}, which is stated without proof.
For completeness, we provide a proof in Section \ref{sec_Proofs}.

%%%%%%%%%%%%%%%%%%%%%%%%%%%%%%%%%%%%%%%%%%%%%%%%%%
%%%%%%%%%%%%%%%%%%%%%%%%%%%%%%%%%%%%%%%%%%%%%%%%%%
%%%%%%%%%%%%%%%%%%%%%%%%%%%%%%%%%%%%%%%%%%%%%%%%%%
%%%%%%%%%%%%%%%%%%%%%%%%%%%%%%%%%%%%%%%%%%%%%%%%%%
%%%%%%%%%%%%%%%%%%%%%%%%%%%%%%%%%%%%%%%%%%%%%%%%%%
%%%%%%%%%%%%%%%%%%%%%%%%%%%%%%%%%%%%%%%%%%%%%%%%%%

\section{Preliminaries}\label{sec_Prelim}

We start by recalling some properties of Gauss and Jacobi sums. See \cite{BEW, IR} for further details, noting that we have adjusted results to take into account $\varepsilon(0)=0$, where $\varepsilon$ is the trivial character. We first note that $g(\varepsilon)=-1$. For $\chi \in \widehat{\mathbb{F}^{*}_{q}}$,
\begin{equation}\label{for_GaussConj}
g(\chi)g(\bar{\chi})=
\begin{cases}
\chi(-1) q & \textup{if } \chi \neq \varepsilon,\\
1 & \textup{if } \chi= \varepsilon.
\end{cases}
\end{equation}

We now state the Hasse-Davenport product formula for Gauss sums.
\begin{theorem}[\cite{BEW} Thm.~11.3.5]\label{thm_HD}
Let $n$ be a positive integer and let $\chi_n \in \widehat{\mathbb{F}_q^{*}}$ be a character of order $n$. For $\psi \in \widehat{\mathbb{F}_q^{*}}$, we have
\begin{equation*}
\prod_{l=0}^{n-1} g(\chi_n^l \psi) = g(\psi^n) \, \psi^{-n}(n)\prod_{l=1}^{n-1} g(\chi_n^l).
\end{equation*}
\end{theorem}

The following is a variant of the standard orthogonal relation, but we prove it here for completeness.
\begin{prop}\label{prop_Orth}
Let $n$ be a positive integer and let $q \equiv 1 \imod n$ be a prime power. Let $\zeta_n$ be a primitive $n$-th root of unity in $\mathbb{F}_q^{*}$ and let $\chi \in \widehat{\mathbb{F}_q^{*}}$. Then
\begin{equation*}
\sum_{l=0}^{n-1} \chi(\zeta_n^l)
=
\begin{cases}
n & \textup{if } \chi = \psi^n \textup{ for some } \psi \in \widehat{\mathbb{F}_q^{*}},\\
0 & \textup{otherwise.}
\end{cases}
\end{equation*}
\end{prop}

\begin{proof}
If $\chi = \psi^n$ then $ \chi(\zeta_n^l) = \psi(\zeta_n^{nl}) = \psi(1) = 1$, and the result follows.
Now assume $\chi$ is not an $n$-th power. Then $\chi(\zeta_n) \neq 1$ and so
\begin{align*}
\chi(\zeta_n) \sum_{l=0}^{n-1} \chi(\zeta_n^l)
=
\sum_{l=0}^{n-1} \chi(\zeta_n^{l+1})
=
\sum_{l=0}^{n-1} \chi(\zeta_n^l)
\end{align*}
implies $\sum_{l=0}^{n-1} \chi(\zeta_n^l)$ must equal zero.
\end{proof}
 
\begin{cor}\label{cor_Orth}
Let $n$ be a positive integer and let $\xi_n \in \mathbb{C}$ be a primitive $n$-th root of unity. 
Then, for a non-negative integer $k$,
\begin{equation*}
\sum_{l=0}^{n-1}
\xi_n^{lk}
=
\begin{cases}
n & \textup{if } k \equiv 0 \imod n,\\
0 & \textup{otherwise.}
\end{cases}
\end{equation*}
\end{cor}

\begin{proof}
Let $q\equiv 1 \imod n$, $\mathbb{F}_q^{*} = \langle \alpha \rangle$ and $\xi_{q-1} \in \mathbb{C}$ be a primitive $(q-1)$-st root of unity such that $\xi_{q-1}^{\frac{q-1}{n}} = \xi_n$. Consider the primitive character $T \in \widehat{\mathbb{F}_q^*}$ defined by $T(\alpha^t) = \xi_{q-1}^t$. Then $\alpha^{\frac{q-1}{n}}$ is a primitive $n$-th root of unity in $\mathbb{F}_q^{*}$ and $T(\alpha^{\frac{q-1}{n}})=\xi_n$. Now let $\chi= T^k$ and $\zeta_n=\alpha^{\frac{q-1}{n}}$ in Proposition \ref{prop_Orth} and the result follows.
\end{proof}

We now recall some details about the Teichm\"{u}ller character and the $p$-adic gamma function. For further details, see \cite{Ko, Ko2}.
Let $\mathbb{Z}_p$ denote the ring of $p$-adic integers, $\mathbb{Q}_p$ the field of $p$-adic numbers, $\bar{\mathbb{Q}_p}$ the algebraic closure of $\mathbb{Q}_p$, and $\mathbb{C}_p$ the completion of $\bar{\mathbb{Q}_p}$. 
Let $\mathbb{Z}_q$ be the ring of integers in the unique unramified extension of $\mathbb{Q}_p$ with residue field $\mathbb{F}_q$.
Recall that for each $x \in \mathbb{F}_q^{*}$, there is a unique Teichm\"{u}ller representative ${\omega(x) \in \mathbb{Z}^{\times}_q}$ such that $\omega(x)$ is a $({q-1})$-st root of unity and $\omega(x) \equiv x \imod p$. Therefore, we define the Teichm\"{u}ller character to be the primitive character $\omega:  \mathbb{F}_q^{*}  \rightarrow\mathbb{Z}^{\times}_q$ given by $x \mapsto \omega(x)$, which we extend with $\omega(0):=0$.

Let $p$ be an odd prime.  For $n \in \mathbb{Z}^{+}$ we define the $p$-adic gamma function as
\begin{align*}
\gfp{n} &:= {(-1)}^n \prod_{\substack{0<j<n\\p \nmid j}} j 
\end{align*}
and extend it to all $x \in\mathbb{Z}_p$ by setting $\gfp{0}:=1$ and
$\gfp{x} := \lim_{n \rightarrow x} \gfp{n}$
for $x\neq 0$, where $n$ runs through any sequence of positive integers $p$-adically approaching $x$. 
This limit exists, is independent of how $n$ approaches $x$, and determines a continuous function
on $\mathbb{Z}_p$ with values in $\mathbb{Z}^{*}_p$.
The function satisfies the following product formula.
\begin{theorem}[\cite{GK} Thm.~3.1]\label{thm_GrossKoblitzMult}
If $n\in\mathbb{Z}^{+}$, $p \nmid n$ and $0\leq x <1$ with $(q-1)x \in \mathbb{Z}$, then
\begin{equation}\label{for_pGammaMult}
\prod_{k=0}^{r-1} \prod_{h=0}^{n-1} \gfp{\langle \tfrac{x+h}{n} p^k \rangle} 
= \omega \big(n^{(q-1)x}\bigr)
\prod_{k=0}^{r-1} \gfp{\langle x p^k \rangle} \prod_{h=1}^{n-1} \gfp{\langle \tfrac{h}{n} p^k \rangle}.
\end{equation}
\end{theorem}
\noindent We note that in the original statement of Theorem \ref{thm_GrossKoblitzMult} in \cite{GK}, $\omega$ is the Teichm\"{u}ller character of $\mathbb{F}_p^{*}$. However, the result above still holds as $\omega|_{\mathbb{F}_p^{*}}$ is the Teichm\"{u}ller character of $\mathbb{F}_p^{*}$.

The Gross-Koblitz formula, Theorem \ref{thm_GrossKoblitz} below, allows us to relate Gauss sums and the $p$-adic gamma function. Let $\pi_p \in \mathbb{C}_p$ be the fixed root of $x^{p-1}+p=0$ that satisfies ${\pi_p \equiv \zeta_p-1 \imod{{(\zeta_p-1)}^2}}$.
\begin{theorem}[\cite{GK} Thm.~1.7]\label{thm_GrossKoblitz}
For $ j \in \mathbb{Z}$,
\begin{equation*} 
g(\bar{\omega}^j)=-\pi_p^{(p-1) \sum_{k=0}^{r-1} \langle{\frac{jp^k}{q-1}}\rangle} \: \prod_{k=0}^{r-1} \gfp{\langle{\tfrac{jp^k}{q-1}}\rangle}.
\end{equation*}
\end{theorem}

The $p$-adic gamma function also satisfies the reflection formula
\begin{equation}\label{for_pGammaOneMinus}
\gfp{x}\gfp{1-x} = {(-1)}^{x_0},
\end{equation}
where $x_0 \in \{1,2, \dotsc, {p}\}$ such that $x_0 \equiv x \pmod {p}$.
Techniques in \cite{GK} allow us to calculate $x_0$ when $x=\langle a p^j \rangle$.
\begin{lemma}[\cite{GK} Lemmas 2.4, 2.5 \& 2.11]\label{lem_GK0}
Let $a \in \mathbb{Q} \cap \mathbb{Z}_p$ with $0 <  a < 1$. 
Let $f \in \mathbb{Z}^{+}$ be such that 
$(p^f-1)a \in \mathbb{Z}^{+}$.
If we write
$$(p^f-1) a = z_f + z_1 \, p +z_2 \, p^2 + \cdots + z_{f-1} \, p^{f-1},$$
for integers $z_i$ with $0\leq z_i < p$,
then
$${\langle a p^j \rangle}_0 = p - z_{f-j},$$
for all $0 \leq j < f$,
and
$$\lfloor a p^j \rfloor = z_{f-j} + z_{f-j+1} \, p + \cdots + z_{f-1} \, p^{j-1},$$
for all $1 \leq j < f$.
Furthermore, 
$$\sum_{i=1}^{f} z_i \equiv (p^f-1) a \pmod{p-1}$$
and
$$\sum_{i=1}^{j} z_{f-i} \equiv \lfloor a p^j \rfloor \pmod{p-1}.$$
\end{lemma}

\begin{cor}\label{cor_GK0}
With the notation of Lemma \ref{lem_GK0},
\begin{equation*}
\sum_{k=0}^{r-1} 
{\langle a p^k \rangle}_0
\equiv 
r -(p^f-1)a + \lfloor a p^{f-1} \rfloor - \lfloor a p^{r-1} \rfloor
\pmod{p-1}.
\end{equation*}
\end{cor}

\begin{proof}
If $r=1$ then both sides are congruent to $1-z_f$. Let $r>1$.
From Lemma \ref{lem_GK0} we see that
\begin{align*}
\sum_{k=0}^{r-1} 
{\langle a p^k \rangle}_0
&=
rp - \sum_{k=0}^{r-1} z_{f-k}\\
&=
rp - \left( \sum_{i=1}^{f} z_i - \sum_{i=1}^{f-1} z_{f-i} \right) - \sum_{k=1}^{r-1} z_{f-k} \\
&\equiv
r - ((p^f-1) a  - \lfloor a p^{f-1} \rfloor) - \lfloor a p^{r-1} \rfloor
\pmod{p-1}.
\end{align*}
\end{proof}

We will need the following result in the proof of Theorem \ref{thm_GQ}.

\begin{lemma}\label{lem_tl}
Let $l\geq 3$ be an integer. 
Let $p$ be an odd prime with $\gcd(p,l)=1$. 
Let $q=p^r$ such that $q \not\equiv 1 \imod l$.
Let $f^{\prime}$ be a positive integer such that $q^{f^{\prime}} \equiv 1 \imod l$.
Let $f= r f^{\prime}$.
Then for all integers $j$,
\begin{equation*}
\sum_{\substack{1 \leq t < l \\ \gcd(t,l)=1}}
\lfloor \langle \tfrac{t}{l} - \tfrac{j}{q-1} \rangle p^{f-1} \rfloor - \lfloor \langle \tfrac{t}{l} - \tfrac{j}{q-1} \rangle p^{r-1} \rfloor
\equiv
0 \pmod 2.
\end{equation*}
\end{lemma}

\begin{proof}
Let $\phi(\cdot)$ denote Euler's totient function. We note that $\phi(l)$ is even, as $l \geq 3$.
For a given $t$, let
$\langle \tfrac{t}{l} - \tfrac{j}{q-1} \rangle =  \tfrac{t}{l} - \tfrac{j}{q-1} + n_{t,j}.$
Note $ n_{t,j} \in \mathbb{Z}$.
So,
\begin{align}\label{for1a}
\notag
\sum_{\substack{1 \leq t < l \\ \gcd(t,l)=1}}
& \lfloor \langle \tfrac{t}{l} - \tfrac{j}{q-1} \rangle p^{f-1} \rfloor - \lfloor \langle \tfrac{t}{l} - \tfrac{j}{q-1} \rangle p^{r-1} \rfloor
\\ & \notag =
\sum_{\substack{1 \leq t < l \\ \gcd(t,l)=1}}
\lfloor  \bigl( \tfrac{t}{l} - \tfrac{j}{q-1} + n_{t,j} \bigr) p^{f-1} \rfloor - \lfloor \bigl( \tfrac{t}{l} - \tfrac{j}{q-1} + n_{t,j} \bigr)  p^{r-1} \rfloor
\\ & \notag =
\sum_{\substack{1 \leq t < l \\ \gcd(t,l)=1}}
\lfloor  \bigl( \tfrac{t}{l} - \tfrac{j}{q-1} \bigr) p^{f-1} \rfloor - \lfloor \bigl( \tfrac{t}{l} - \tfrac{j}{q-1}  \bigr)  p^{r-1} \rfloor 
+
p^{r-1} (p^{f-r} - 1)
\sum_{\substack{1 \leq t < l \\ \gcd(t,l)=1}}
n_{t,j}
\\ & \equiv
\sum_{\substack{1 \leq t < l \\ \gcd(t,l)=1}}
\lfloor  \bigl( \tfrac{t}{l} - \tfrac{j}{q-1} \bigr) p^{f-1} \rfloor - \lfloor \bigl( \tfrac{t}{l} - \tfrac{j}{q-1}  \bigr)  p^{r-1} \rfloor 
\pmod{2}.
\end{align}
as $p^{f-r} - 1$ is even.
Now,
\begin{align}\label{for1b}
\notag
\sum_{\substack{1 \leq t < l \\ \gcd(t,l)=1}}
& \lfloor  \bigl( \tfrac{t}{l} - \tfrac{j}{q-1} \bigr) p^{f-1} \rfloor - \lfloor \bigl( \tfrac{t}{l} - \tfrac{j}{q-1}  \bigr)  p^{r-1} \rfloor 
\\ \notag & =
\sum_{\substack{1 \leq t < l \\ \gcd(t,l)=1}}
\bigl( \tfrac{t}{l} - \tfrac{j}{q-1} \bigr) p^{f-1} 
-
\langle  \bigl( \tfrac{t}{l} - \tfrac{j}{q-1} \bigr) p^{f-1} \rangle 
- 
\bigl( \tfrac{t}{l} - \tfrac{j}{q-1}  \bigr)  p^{r-1}
+
\langle \bigl( \tfrac{t}{l} - \tfrac{j}{q-1}  \bigr)  p^{r-1} \rangle
\\ \notag & =
p^{r-1}  (p^{f-r} - 1)
\underbrace{\sum_{\substack{1 \leq t < l \\ \gcd(t,l)=1}} \tfrac{t}{l}}_{=\frac{\phi(l)}{2} \, \in \, \mathbb{Z}}
-
\, j 
p^{r-1}
\left(\tfrac{q^{f^{\prime}-1}-1}{q-1}\right)
\underbrace{\sum_{\substack{1 \leq t < l \\ \gcd(t,l)=1}} 1}_{=\phi(l) \, \in \, 2 \mathbb{Z}}
\\ \notag & \qquad \qquad \qquad \qquad \qquad \qquad \qquad \quad
-
\sum_{\substack{1 \leq t < l \\ \gcd(t,l)=1}} 
\langle  \bigl( \tfrac{t}{l} - \tfrac{j}{q-1} \bigr) p^{f-1} \rangle 
-
\langle \bigl( \tfrac{t}{l} - \tfrac{j}{q-1}  \bigr)  p^{r-1} \rangle
\\ & \equiv
\sum_{\substack{1 \leq t < l \\ \gcd(t,l)=1}} 
\langle  \bigl( \tfrac{t}{l} - \tfrac{j}{q-1} \bigr) p^{f-1} \rangle 
-
\langle \bigl( \tfrac{t}{l} - \tfrac{j}{q-1}  \bigr)  p^{r-1} \rangle
\pmod{2}.
\end{align}
So, combining (\ref{for1a}) and (\ref{for1b}), it suffices to prove that
\begin{equation*}\label{for1c}
\sum_{\substack{1 \leq t < l \\ \gcd(t,l)=1}} 
\langle  \bigl( \tfrac{t}{l} - \tfrac{j}{q-1} \bigr) p^{f-1} \rangle 
-
\langle \bigl( \tfrac{t}{l} - \tfrac{j}{q-1}  \bigr)  p^{r-1} \rangle
\equiv 0
\pmod{2}.
\end{equation*}
Now,
\begin{multline}\label{for1}
\sum_{\substack{1 \leq t < l \\ \gcd(t,l)=1}} 
\langle  \bigl( \tfrac{t}{l} - \tfrac{j}{q-1} \bigr) p^{f-1} \rangle 
-
\langle \bigl( \tfrac{t}{l} - \tfrac{j}{q-1}  \bigr)  p^{r-1} \rangle
\\ =
\sum_{\substack{1 \leq t < l \\ \gcd(t,l)=1}} 
\left(
\langle  \tfrac{t}{l} \, p^{f-1} \rangle 
-
\langle \tfrac{j}{q-1} \, p^{f-1} \rangle 
+
\delta_{t,j}(f)
\right)
-
\left(
\langle  \tfrac{t}{l} \, p^{r-1} \rangle 
-
\langle \tfrac{j}{q-1} \, p^{r-1} \rangle 
+
\delta_{t,j}(r)
\right),
\end{multline}
where
\begin{equation*}
\delta_{t,j}(a)
=
\begin{cases}
0 & \textup{if } \langle  \tfrac{t}{l} \, p^{a-1} \rangle \geq \langle \tfrac{j}{q-1} \, p^{a-1} \rangle,\\ 
1 & \textup{if } \langle  \tfrac{t}{l} \, p^{a-1} \rangle < \langle \tfrac{j}{q-1} \, p^{a-1} \rangle.
\end{cases}
\end{equation*}
We note that
\begin{equation}\label{for2}
\langle \tfrac{j}{q-1} \, p^{f-1} \rangle
=
\langle j \, \tfrac{q^{f^{\prime}-1}-1}{q-1} \, p^{r-1} + \tfrac{j}{q-1} \,  p^{r-1}  \rangle
 =
\langle \tfrac{j}{q-1} \,  p^{r-1}  \rangle.
\end{equation}
Also, as $\gcd(p,l)=1$, we have that for any non-negative integer $a$,
\begin{equation}\label{for3}
\{t \mid 1 \leq t < l, \gcd(t,l)=1 \} \stackrel{\pmod l}{\equiv} \{ t \, p^a \mid 1 \leq t < l, \gcd(t,l)=1 \}
\end{equation}
and hence
\begin{equation}\label{for4}
\sum_{\substack{1 \leq t < l \\ \gcd(t,l)=1}} 
\langle  \tfrac{t}{l} \, p^{f-1} \rangle 
=
\sum_{\substack{1 \leq t < l \\ \gcd(t,l)=1}} 
\langle  \tfrac{t}{l} \, p^{r-1} \rangle.
\end{equation}
Then, combining (\ref{for2}) and (\ref{for3}) we get that 
\begin{align}\label{for5}
\sum_{\substack{1 \leq t < l \\ \gcd(t,l)=1}} 
\notag 
\delta_{t,j}(f)
& =
| \{t \mid 1 \leq t < l, \gcd(t,l)=1, \langle  \tfrac{t}{l} \, p^{f-1} \rangle < \langle \tfrac{j}{q-1} \, p^{f-1} \rangle \}|
\\ \notag & =
| \{t \mid 1 \leq t < l, \gcd(t,l)=1, \langle  \tfrac{t}{l} \, p^{f-1} \rangle < \langle \tfrac{j}{q-1} \, p^{r-1} \rangle \}|
\\ \notag & =
| \{t \mid 1 \leq t < l, \gcd(t,l)=1, \langle  \tfrac{t}{l} \rangle < \langle \tfrac{j}{q-1} \, p^{r-1} \rangle \}|
\\ \notag & =
| \{t \mid 1 \leq t < l, \gcd(t,l)=1, \langle  \tfrac{t}{l} \, p^{r-1} \rangle < \langle \tfrac{j}{q-1} \, p^{r-1} \rangle \}|
\\ & =
\sum_{\substack{1 \leq t < l \\ \gcd(t,l)=1}} 
\delta_{t,j}(r)
\end{align}
Accounting for (\ref{for2}), (\ref{for4}) and (\ref{for5}) in (\ref{for1}) completes the proof.
\end{proof}

%%%%%%%%%%%%%%%%%%%%%%%%%%%%%%%%%%%%%%%%%%%%%%%%%%
%%%%%%%%%%%%%%%%%%%%%%%%%%%%%%%%%%%%%%%%%%%%%%%%%%
%%%%%%%%%%%%%%%%%%%%%%%%%%%%%%%%%%%%%%%%%%%%%%%%%%
%%%%%%%%%%%%%%%%%%%%%%%%%%%%%%%%%%%%%%%%%%%%%%%%%%
%%%%%%%%%%%%%%%%%%%%%%%%%%%%%%%%%%%%%%%%%%%%%%%%%%
%%%%%%%%%%%%%%%%%%%%%%%%%%%%%%%%%%%%%%%%%%%%%%%%%%

\section{Proofs of Main Results}\label{sec_Proofs}
\subsection{Finite Field Setting}\label{subsec_Proofs_FF}
We first prove Theorems \ref{thm_F_Main} \& \ref{thm_F_Main_Converse}.
\begin{proof}[Proof of Theorem \ref{thm_F_Main}]
We expand the summands by definition and then apply Proposition \ref{prop_Orth} and  Theorem \ref{thm_HD}, as follows:
\begin{align*}
\sum_{l=0}^{n-1}
{_{m}F_{m}} & {\biggl( \begin{array}{cccc} 
A_1^n, & A_2^n, & \dotsc, & A_m^n \\[1pt]
B_1^n, & B_2^n, & \dotsc, & B_m^n \end{array}
\Big| \; \zeta_n^l \cdot \lambda \biggr)}_{q}\\
&=
\frac{-1}{q-1} 
\sum_{l=0}^{n-1}
\sum_{\chi \in \widehat{\mathbb{F}_q^{*}}}
\prod_{i=1}^{m} \frac{g(A_i^n \chi)}{g(A_i^n)} \frac{g(\bar{B_i^n \chi})}{g(\bar{B_i^n})}
\chi(-1)^{m}
\chi(\zeta_n^l \cdot \lambda)\\
&=
\frac{-1}{q-1}  
\sum_{\chi \in \widehat{\mathbb{F}_q^{*}}}
\prod_{i=1}^{m} \frac{g(A_i^n \chi)}{g(A_i^n)} \frac{g(\bar{B_i^n \chi})}{g(\bar{B_i^n})}
\chi(-1)^{m}
\chi(\lambda)
\sum_{l=0}^{n-1}
\chi(\zeta_n^l)\\
&=
\frac{-n}{q-1}  
\sum_{\substack{\chi \in \widehat{\mathbb{F}_q^{*}}\\ \chi=\psi^n}}
\prod_{i=1}^{m} \frac{g(A_i^n \chi)}{g(A_i^n)} \frac{g(\bar{B_i^n \chi})}{g(\bar{B_i^n})}
\chi(-1)^{m}
\chi(\lambda)\\
&=
\frac{-1}{q-1}  
\sum_{\psi \in \widehat{\mathbb{F}_q^{*}}}
\prod_{i=1}^{m} \frac{g(A_i^n \psi^n)}{g(A_i^n)} \frac{g(\bar{B_i^n \psi^n})}{g(\bar{B_i^n})}
\psi(-1)^{nm}
\psi(\lambda^n)\\
&=
\frac{-1}{q-1}  
\sum_{\psi \in \widehat{\mathbb{F}_q^{*}}}
\prod_{i=1}^{m}
\prod_{l=0}^{n-1}
\frac{g(\chi_n^l A_i \psi)}{g(\chi_n^l A_i)} 
\frac{g(\bar{\chi_n^l B_i \psi})}{g(\bar{\chi_n^l B_i})}
\psi(-1)^{nm}
\psi(\lambda^n)\\
&=
{_{nm}F_{nm}} {\biggl( \begin{array}{c} 
A_i \, \chi_n^l \, : 1 \leq i \leq m, 0 \leq l \leq n-1\\[1pt]
B_i \, \chi_n^l \, : 1 \leq i \leq m, 0 \leq l \leq n-1 \end{array}
\Big| \; \lambda^n \biggr)}_{q}.
\end{align*}
\end{proof}

\begin{proof}[Proof of Theorem \ref{thm_F_Main_Converse}]
By Definition \ref{def_F}, we have
\begin{multline}\label{proof_F_Main_Converse_for1}
{_{nm}F_{nm}} {\biggl( \begin{array}{c} 
A_i \,\chi_n^l \, : 1 \leq i \leq m, 0 \leq l \leq n-1  \\[1pt]
B_i \, \chi_n^l  \, : 1 \leq i \leq m, 0 \leq l \leq n-1 \end{array}
\Big| \; \lambda \biggr)}_{q}\\
=
\frac{-1}{q-1}  
\sum_{\chi \in \widehat{\mathbb{F}_q^{*}}}
\prod_{i=1}^{m}
\prod_{l=0}^{n-1}
\frac{g(\chi_n^l A_i \chi)}{g(\chi_n^l A_i)} 
\frac{g(\bar{\chi_n^l B_i \chi})}{g(\bar{\chi_n^l B_i})}
\chi(-1)^{nm}
\chi(\lambda).
\end{multline}
Making the change of variable $\chi \rightarrow \chi\chi_n$ in the right hand side of (\ref{proof_F_Main_Converse_for1}) gives us
\begin{multline*}
{_{nm}F_{nm}} {\biggl( \begin{array}{c} 
A_i \,\chi_n^l \, : 1 \leq i \leq m, 0 \leq l \leq n-1  \\[1pt]
B_i \, \chi_n^l  \, : 1 \leq i \leq m, 0 \leq l \leq n-1 \end{array}
\Big| \; \lambda \biggr)}_{q}\\
=
\chi_n(\lambda) \cdot
{_{nm}F_{nm}} {\biggl( \begin{array}{c} 
A_i \,\chi_n^l \, : 1 \leq i \leq m, 0 \leq l \leq n-1  \\[1pt]
B_i \, \chi_n^l  \, : 1 \leq i \leq m, 0 \leq l \leq n-1 \end{array}
\Big| \; \lambda \biggr)}_{q}.
\end{multline*}
Using the fact that $\chi_n(\lambda)=1$ if and only if $\lambda$ is an $n$-th power, completes the proof.
\end{proof}

%%%%%%%%%%%%%%%%%%%%%%%%%%%%%%%%%%%%%%%%%%%%%%%%%%
%%%%%%%%%%%%%%%%%%%%%%%%%%%%%%%%%%%%%%%%%%%%%%%%%%
%%%%%%%%%%%%%%%%%%%%%%%%%%%%%%%%%%%%%%%%%%%%%%%%%%

\subsection{$p$-adic Setting}\label{subsec_Proofs_padic}
We now prove Lemma \ref{lem_G_to_F} and Theorems \ref{thm_GQ}-\ref{thm_G_Main_Converse}.
\begin{proof}[Proof of Lemma \ref{lem_G_to_F}]
The case when $q=p$ was proven in \cite[Lemma 3.3]{McC7} and the same approach works here. 
If we let $\chi= \omega^j$, $A_i=\bar{\omega}^{a_i(q-1)}$ and $B_i=\bar{\omega}^{b_i(q-1)}$, then it is straightforward to show, using the Gross-Koblitz formula, Theorem \ref{thm_GrossKoblitz}, that
\begin{multline}\label{proof_lem_G_to_F_for1}
\frac{g(A_i \chi)}{g(A_i)}
\frac{g(\bar{B_i \chi})}{g(\bar{B_i})}\\
=
\prod_{k=0}^{r-1} 
\frac{\biggfp{\langle (a_i -\frac{j}{q-1})p^k \rangle}}{\biggfp{\langle a_i p^k \rangle}}
\frac{\biggfp{\langle (-b_i +\frac{j}{q-1}) p^k \rangle}}{\biggfp{\langle -b_i p^k\rangle}}
(-p)^{-\lfloor{\langle a_i p^k \rangle -\frac{j p^k}{q-1}}\rfloor-\lfloor{\langle -b_i p^k\rangle +\frac{j p^k}{q-1}}\rfloor}.
\end{multline}
Substituting for (\ref{proof_lem_G_to_F_for1}) and $\chi= \omega^j$ in Definition \ref{def_Gq} yields the result.
\end{proof}

\begin{proof}[Proof of Theorem \ref{thm_GQ}]
Let $(\{a_i\}, \{b_i\})$ be defined over $\mathbb{Q}$ with corresponding exponents $(\{p_i : 1 \leq i \leq t \}, \{q_i  : 1 \leq i \leq s \})$, i.e.,
$$\prod_{i=1}^{m} \frac{x- e^{2 \pi \I a_i}}{x- e^{2 \pi \I b_i}} = \frac{\prod_{i=1}^{t} x^{p_i} - 1}{\prod_{i=1}^{s} x^{q_i} - 1},$$
with $t$ and $s$ minimal. Recall, $D(x) := \gcd({\prod_{i=1}^{t} x^{p_i} - 1},{\prod_{i=1}^{s} x^{q_i} - 1})$, of degree $\delta$, has zeros $\exp \big(\frac{2 \pi \I \, c_h}{q-1}\bigr)$ for $1 \leq h \leq \delta$ and $\mathfrak{s}(c)$ denotes the multiplicity of the zero $\exp \bigl(\frac{2 \pi \I \, c}{q-1}\bigr)$.
Then
\begin{align}\label{thm_GQ_for1}
\notag 
\prod_{i=1}^{m} &
\frac{\biggfp{\langle (a_i -\frac{j}{q-1} )p^k \rangle}}{\biggfp{\langle a_i p^k \rangle}}
\frac{\biggfp{\langle (-b_i +\frac{j}{q-1}) p^k \rangle}}{\biggfp{\langle -b_i p^k\rangle}}
(-p)^{-\lfloor{\langle a_i p^k \rangle -\frac{j p^k}{q-1}}\rfloor -\lfloor{\langle -b_i p^k\rangle +\frac{j p^k}{q-1}}\rfloor}\\
\notag 
&=
\prod_{i=1}^{t} 
\prod_{h=0}^{p_i-1} 
\frac{\biggfp{\langle (\frac{h}{p_i} -\frac{j}{q-1} )p^k \rangle}}{\biggfp{\langle \frac{h}{p_i} p^k \rangle}}
(-p)^{-\lfloor{\langle \frac{h}{p_i} p^k \rangle -\frac{j p^k}{q-1}}\rfloor}\\
\notag
&\quad \times
\prod_{i=1}^{s} 
\prod_{h=0}^{q_i-1} 
\frac{\biggfp{\langle (-\frac{h}{q_i} + \frac{j}{q-1} )p^k \rangle}}{\biggfp{\langle -\frac{h}{q_i} p^k \rangle}}
(-p)^{-\lfloor{\langle -\frac{h}{q_i}  p^k\rangle +\frac{j p^k}{q-1}}\rfloor}\\
&\quad \times
\prod_{h=1}^{\delta}
\frac{\biggfp{\langle \frac{c_h}{q-1} p^k \rangle}}{\biggfp{\langle (\frac{c_h}{q-1} -\frac{j}{q-1} )p^k \rangle}}
\frac{\biggfp{\langle -\frac{c_h}{q-1} p^k \rangle}}{\biggfp{\langle (-\frac{c_h}{q-1} +\frac{j}{q-1} )p^k \rangle}}
(-p)^{\lfloor{\langle \frac{c_h}{q-1} p^k \rangle -\frac{j p^k}{q-1}}\rfloor + \lfloor{\langle -\frac{c_h}{q-1} p^k\rangle +\frac{j p^k}{q-1}}\rfloor}.
\end{align}
For a given $i$, with $1\leq i \leq t$, choose $l \in \mathbb{Z}$ such that 
$0 \leq l - \frac{j p_i}{q-1} < 1$.
Then, noting that $\{l+h \mid h=0,\dots, p_i-1\} \equiv \{h \mid h=0,\dots,p_i-1\} \imod{p_i}$ and applying Theorem \ref{thm_GrossKoblitzMult} with $n= p_i$ and $x = l - \frac{j p_i}{q-1}$ we get that
\begin{align}\label{thm_GQ_for2}
\notag 
\prod_{k=0}^{r-1} 
\prod_{h=0}^{p_i-1} 
\frac{\biggfp{\langle (\frac{h}{p_i} -\frac{j}{q-1} )p^k \rangle}}{\biggfp{\langle \frac{h}{p_i} p^k \rangle}}
\notag 
&=
\prod_{k=0}^{r-1} 
\prod_{h=0}^{p_i-1} 
\frac{\biggfp{\langle \Bigl(\frac{h - \frac{j p_i}{q-1}}{p_i} \Bigr) p^k \rangle}}{\biggfp{\langle \frac{h}{p_i} p^k \rangle}}\\
\notag 
&=
\prod_{k=0}^{r-1} 
\prod_{h=0}^{p_i-1} 
\frac{\biggfp{\langle \Bigl(\frac{l + h - \frac{j p_i}{q-1}}{p_i} \Bigr) p^k \rangle}}{\biggfp{\langle \frac{h}{p_i} p^k \rangle}}\\
\notag 
&=
\omega \Bigl(p_i^{(q-1)(l - \frac{j p_i}{q-1})}\Bigr)
\prod_{k=0}^{r-1} \gfp{\langle \bigl(l - \tfrac{j p_i}{q-1}\bigr) p^k \rangle}\\
&=
\omega \bigl(p_i^{- j p_i}\bigr)
\prod_{k=0}^{r-1} \gfp{\langle - \tfrac{j p_i}{q-1} p^k \rangle}.
\end{align}
Similarly, with $0 \leq \frac{j q_i}{q-1} -l < 1$, we have
\begin{equation}\label{thm_GQ_for3}
\prod_{k=0}^{r-1} 
\prod_{h=0}^{q_i-1} 
\frac{\biggfp{\langle (-\frac{h}{q_i} + \frac{j}{q-1} )p^k \rangle}}{\biggfp{\langle -\frac{h}{q_i} p^k \rangle}}
=
\omega \bigl(q_i^{ j q_i}\bigr)
\prod_{k=0}^{r-1} \gfp{\langle \tfrac{j q_i}{q-1} p^k \rangle}.
\end{equation}

As $a_i, b_i \in \mathbb{Z}_p$, $p$ does not divide their denominators and, hence, $\gcd(p,p_i)=\gcd(p,q_i)=1$. 
So $\{hp^k \mid h=0,\dots, p_i-1\} \equiv \{h \mid h=0,\dots,p_i-1\} \imod{p_i}$ and  
\begin{align}\label{thm_GQ_for4}
\sum_{h=0}^{p_i-1} \lfloor{\langle \tfrac{h}{p_i} p^k \rangle -\tfrac{j p^k}{q-1}}\rfloor
=
\sum_{h=0}^{p_i-1} \lfloor{\langle \tfrac{h}{p_i} \rangle -\tfrac{j p^k}{q-1}}\rfloor
=
\sum_{h=0}^{p_i-1} \lfloor{ \tfrac{h}{p_i} -\tfrac{j p^k}{q-1}}\rfloor
= 
\lfloor{ \tfrac{-j p_i}{q-1} p^k} \rfloor,
\end{align}
where we have used Hermite's identity:  for a positive integer $m$, $\left\lfloor mx \right\rfloor = \sum_{h=0}^{m-1} \left\lfloor x+ \frac{h}{m}  \right\rfloor$.
Similarly,
\begin{align}\label{thm_GQ_for5}
\sum_{h=0}^{q_i-1} \lfloor{\langle -\tfrac{h}{q_i}  p^k\rangle +\tfrac{j p^k}{q-1}}\rfloor
=
\lfloor{ \tfrac{j q_i}{q-1} p^k} \rfloor.
\end{align}

Accounting for (\ref{thm_GQ_for2})-(\ref{thm_GQ_for5}) in (\ref{thm_GQ_for1}) and substituting the result into Definition \ref{def_Gq}, we see that it now suffices to prove
\begin{multline}\label{thm_GQ_for6}
\prod_{h=1}^{\delta}
\prod_{k=0}^{r-1}
\frac{\biggfp{\langle \frac{c_h}{q-1} p^k \rangle}}{\biggfp{\langle (\frac{c_h}{q-1} -\frac{j}{q-1} )p^k \rangle}}
\frac{\biggfp{\langle -\frac{c_h}{q-1} p^k \rangle}}{\biggfp{\langle (-\frac{c_h}{q-1} +\frac{j}{q-1} )p^k \rangle}}
(-p)^{\lfloor{\langle \frac{c_h}{q-1} p^k \rangle -\frac{j p^k}{q-1}}\rfloor + \lfloor{\langle -\frac{c_h}{q-1} p^k\rangle +\frac{j p^k}{q-1}}\rfloor}\\
=
q^{-\mathfrak{s}(0)+\mathfrak{s}(j)} \, (-1)^{j\delta}.
\end{multline}

$D(x)$ is necessarily the product of some cyclotomic polynomials, i.e.,
$D(x) = \prod_{l \in S} \Phi_l(x)$ for some set $S$ of positive integers coprime to $p$, where $\Phi_l(x)$ denotes the $l$-th cyclotomic polynomial.
Then the zeros of $D(x)$ are necessarily all the zeros of these $\Phi_l(x)$. 
If $l=1$ or $l=2$, then $\Phi_l(x)$ has only one zero corresponding to $\exp \big(\frac{2 \pi \I \, c}{q-1}\bigr)$ with $c=0$ or $c=\frac{q-1}{2}$, respectively, which are both integers. So, if $c_h \not\in \mathbb{Z}$, then $\exp \big(\frac{2 \pi \I \, c_h}{q-1}\bigr)$ is a zero of a cyclotomic polynomial $\Phi_l(x)$ with $l\geq 3$. 
The zeros of $\Phi_l(x)$ are $\{ \exp \big(\frac{2 \pi \I \, t}{l}\bigr) | 1 \leq t < l, \, \gcd(t,l)=1 \}$.
Now,
$\exp \big(\frac{2 \pi \I \, c_h}{q-1}\bigr) = \exp \big(\frac{2 \pi \I \, t}{l}\bigr)$, with $c_h \in \mathbb{Z}$, if and only if $q \equiv 1 \imod l$. 
In which case, all zeros of $\Phi_l(x)$ can be written as $\exp \big(\frac{2 \pi \I \, c}{q-1}\bigr)$ for $c \in \mathbb{Z}$.
Conversely, if $q \not\equiv 1 \imod l$ then all zeros, when written in the form $\exp \big(\frac{2 \pi \I \, c}{q-1}\bigr)$ will have $c \not\in \mathbb{Z}$.
Consequently, if we let $S^{\prime}$ be the subset of $S$ containing all $l$ such that $q \not\equiv 1 \imod l$, then
\begin{equation*}
\left\{\frac{c_h}{q-1} \mid c_h \not\in \mathbb{Z} \right\}
=
\displaystyle \bigcup_{l \in S^{\prime}}
\left\{\frac{t}{l} \mid 1 \leq t < l, \gcd(t,l)=1 \right\}.
\end{equation*}
If we let $\delta = \delta_1 +\delta_2$, where
$\delta_1= |\{h \mid c_h \in \mathbb{Z} \} |$
and 
$\delta_2= |\{h \mid c_h \notin \mathbb{Z} \} |$,
then 
$\delta_2 = \sum_{l \in S^{\prime}} \phi(l)$ is even, as $l \geq 3$ for all $l \in S^{\prime}$
(here $\phi$ is Euler's totient function), and, $\delta \equiv \delta_1 \imod 2$.

We now examine (\ref{thm_GQ_for6}) when $c_h \in \mathbb{Z}$. 
Straightforward applications of the Gross-Koblitz formula (Theorem \ref{thm_GrossKoblitz}) and (\ref{for_GaussConj}) yield
\begin{multline}\label{thm_GQ_for7}
\prod_{\substack{h=1 \\ c_h \in \mathbb{Z}}}^{\delta}
\prod_{k=0}^{r-1}
\frac{\biggfp{\langle \frac{c_h}{q-1} p^k \rangle}}{\biggfp{\langle (\frac{c_h}{q-1} -\frac{j}{q-1} )p^k \rangle}}
\frac{\biggfp{\langle -\frac{c_h}{q-1} p^k \rangle}}{\biggfp{\langle (-\frac{c_h}{q-1} +\frac{j}{q-1} )p^k \rangle}}
(-p)^{\lfloor{\langle \frac{c_h}{q-1} p^k \rangle -\frac{j p^k}{q-1}}\rfloor + \lfloor{\langle -\frac{c_h}{q-1} p^k\rangle +\frac{j p^k}{q-1}}\rfloor}
\\ =
\prod_{\substack{h=1 \\ c_h \in \mathbb{Z}}}^{\delta}
\frac{g(\bar{\omega}^{c_h}) g(\bar{\omega}^{-c_h})}{g(\bar{\omega}^{c_h-j}) g(\bar{\omega}^{-c_h+j})} 
= 
q^{-\mathfrak{s}(0)+\mathfrak{s}(j)} \, \bar{\omega}(-1)^{j\delta_1}
= 
q^{-\mathfrak{s}(0)+\mathfrak{s}(j)} \, (-1)^{j\delta}.
\end{multline}

Next we examine (\ref{thm_GQ_for6}) when $c_h \notin \mathbb{Z}$.
As $c_h \neq 0$, we have $\frac{c_h p^k}{q-1} \not\in \mathbb{Z}$ and so 
$\langle -\frac{c_h p^k}{q-1} \rangle = 1 - \langle \frac{c_h p^k}{q-1} \rangle$.
Similarly, $c_h \neq j$ implies $(\frac{c_h}{q-1} -\frac{j}{q-1} )p^k \not\in \mathbb{Z}$.
Therefore, as $\lfloor x \rfloor  + \lfloor -x \rfloor = -1$ when $x \not\in \mathbb{Z}$,
\begin{align}\label{thm_GQ_for8}
\notag
\lfloor{\langle \tfrac{c_h p^k}{q-1}  \rangle -\tfrac{j p^k}{q-1}}\rfloor + \lfloor{\langle -\tfrac{c_h p^k}{q-1} \rangle +\tfrac{j p^k}{q-1}}\rfloor
&=
\lfloor{\langle \tfrac{c_h p^k}{q-1}  \rangle -\tfrac{j p^k}{q-1}}\rfloor + \lfloor{1 - \langle \tfrac{c_h p^k}{q-1} \rangle +\tfrac{j p^k}{q-1}}\rfloor
\\ \notag & =
\lfloor{\tfrac{c_h p^k}{q-1} - \lfloor \tfrac{c_h p^k}{q-1} \rfloor -\tfrac{j p^k}{q-1}}\rfloor + \lfloor{1 - \tfrac{c_h p^k}{q-1} + \lfloor \tfrac{c_h p^k}{q-1} \rfloor +\tfrac{j p^k}{q-1}}\rfloor
\\ & =
1 + \lfloor{\tfrac{c_h p^k}{q-1} -\tfrac{j p^k}{q-1}}\rfloor + \lfloor{- \tfrac{c_h p^k}{q-1} +\tfrac{j p^k}{q-1}}\rfloor = 0.
\end{align}
%\\ & =
%0.

Applying (\ref{for_pGammaOneMinus}) we see that
\begin{align}\label{thm_GQ_for9}
\notag 
\prod_{\substack{h=1 \\ c_h \notin \mathbb{Z}}}^{\delta}
\prod_{k=0}^{r-1}
&
\biggfp{\langle \tfrac{c_h}{q-1} p^k \rangle}
\biggfp{\langle -\tfrac{c_h}{q-1} p^k \rangle}
\\ \notag &=
\prod_{k=0}^{r-1} \,
\prod_{l \in S^{\prime}}
\prod_{\substack{1 \leq t < l \\ \gcd(t,l)=1}}
\biggfp{\langle \tfrac{t}{l} p^k \rangle}
\biggfp{\langle -\tfrac{t}{l} p^k \rangle}
\\ \notag &=
\prod_{k=0}^{r-1} \,
\prod_{l \in S^{\prime}}
\prod_{\substack{1 \leq t < \frac{l}{2} \\ \gcd(t,l)=1}}
\biggfp{\langle \tfrac{t}{l} p^k \rangle}
\biggfp{\langle -\tfrac{t}{l} p^k \rangle}
\biggfp{\langle (1-\tfrac{t}{l}) p^k \rangle}
\biggfp{\langle -(1-\tfrac{t}{l}) p^k \rangle}
\\ \notag &=
\prod_{k=0}^{r-1} \,
\prod_{l \in S^{\prime}}
\prod_{\substack{1 \leq t < \frac{l}{2} \\ \gcd(t,l)=1}}
\left(
\biggfp{\langle \tfrac{t}{l} p^k \rangle}
\biggfp{\langle -\tfrac{t}{l} p^k \rangle}
\right)^2
\\ &=
\prod_{k=0}^{r-1} \,
\prod_{l \in S^{\prime}}
\prod_{\substack{1 \leq t < \frac{l}{2} \\ \gcd(t,l)=1}}
\left(
\biggfp{\langle \tfrac{t}{l} p^k \rangle}
\biggfp{1 - \langle \tfrac{t}{l} p^k \rangle}
\right)^2 = 1.
%\\ &=
%\prod_{k=0}^{r-1} \,
%\prod_{l \in S^{\prime}}
%\prod_{\substack{1 \leq t < \frac{l}{2} \\ \gcd(t,l)=1}}
%\left(
%(-1)^{{\langle \tfrac{t}{l} p^k \rangle}_0}
%\right)^2
%=1
\end{align}

Similarly,
\begin{align}\label{thm_GQ_for10}
\notag
\prod_{\substack{h=1 \\ c_h \notin \mathbb{Z}}}^{\delta}
\prod_{k=0}^{r-1}
&
\biggfp{\langle (\tfrac{c_h}{q-1} -\tfrac{j}{q-1} )p^k \rangle}
\biggfp{\langle (-\tfrac{c_h}{q-1} +\tfrac{j}{q-1} )p^k \rangle}
\\ \notag &=
\prod_{k=0}^{r-1} \,
\prod_{l \in S^{\prime}}
\prod_{\substack{1 \leq t < l \\ \gcd(t,l)=1}}
\biggfp{\langle (\tfrac{t}{l} -\tfrac{j}{q-1} )p^k \rangle}
\biggfp{\langle (-\tfrac{t}{l} +\tfrac{j}{q-1} )p^k \rangle}
\\ \notag &=
\prod_{k=0}^{r-1} \,
\prod_{l \in S^{\prime}}
\prod_{\substack{1 \leq t < l \\ \gcd(t,l)=1}}
\biggfp{\langle (\tfrac{t}{l} -\tfrac{j}{q-1} )p^k \rangle}
\biggfp{1 - \langle (\tfrac{t}{l} -\tfrac{j}{q-1} )p^k \rangle}
\\ &=
\prod_{k=0}^{r-1} \,
\prod_{l \in S^{\prime}}
\prod_{\substack{1 \leq t < l \\ \gcd(t,l)=1}}
(-1)^{{\langle (\tfrac{t}{l} -\tfrac{j}{q-1} )p^k \rangle}_0}.
\end{align}

For a given $l \in S^{\prime}$, let $f^{\prime} \in \mathbb{Z}^{+}$ be chosen such that $q^{f^{\prime}} \equiv 1 \imod l$. We note that $f^{\prime}>1$. We let $f = r f^{\prime}$. Then $(p^f-1) {\langle \tfrac{t}{l} -\tfrac{j}{q-1} \rangle} \in \mathbb{Z}$.
So, by Corollary \ref{cor_GK0}, with $a = \langle \tfrac{t}{l} - \tfrac{j}{q-1} \rangle$, and Lemma \ref{lem_tl} we have that
\begin{align}\label{thm_GQ_for11}
\notag
\sum_{\substack{1 \leq t < l \\ \gcd(t,l)=1}}
&
\sum_{k=0}^{r-1}
\bigl( \langle (\tfrac{t}{l} - \tfrac{j}{q-1} )p^k \rangle \bigr)_0
\\ \notag & =
\sum_{\substack{1 \leq t < l \\ \gcd(t,l)=1}}
\sum_{k=0}^{r-1}
\bigl( \langle \langle \tfrac{t}{l} - \tfrac{j}{q-1} \rangle p^k \rangle \bigr)_0
\\ \notag & \equiv
%\\ \notag & \stackrel{\imod2}{\equiv}
\sum_{\substack{1 \leq t < l \\ \gcd(t,l)=1}}
r -(p^f-1)\langle \tfrac{t}{l} - \tfrac{j}{q-1} \rangle + \lfloor \langle \tfrac{t}{l} - \tfrac{j}{q-1} \rangle p^{f-1} \rfloor - \lfloor \langle \tfrac{t}{l} - \tfrac{j}{q-1} \rangle p^{r-1} \rfloor
\\ \notag & \equiv
\phi(l) \, r - \sum_{\substack{1 \leq t < l \\ \gcd(t,l)=1}} (p^f-1)\langle \tfrac{t}{l} - \tfrac{j}{q-1} \rangle
%\\ \notag & \equiv
%\sum_{\substack{1 \leq t < l \\ \gcd(t,l)=1}} (p^f-1)\langle \tfrac{t}{l} - \tfrac{j}{q-1} \rangle
\\ \notag & \equiv
\sum_{\substack{1 \leq t < \frac{l}{2} \\ \gcd(t,l)=1}}
(p^f-1) 
\left( \langle \tfrac{t}{l} - \tfrac{j}{q-1} \rangle
+
\langle 1 -\tfrac{t}{l} - \tfrac{j}{q-1} \rangle \right)
\\ \notag & \equiv
\sum_{\substack{1 \leq t < \frac{l}{2} \\ \gcd(t,l)=1}}
(p^f-1) 
\left(
\langle \tfrac{t}{l} - \tfrac{j}{q-1} \rangle
+
\langle -\tfrac{t}{l} - \tfrac{j}{q-1} \rangle
\right)
\\ \notag & \equiv
\sum_{\substack{1 \leq t < \frac{l}{2} \\ \gcd(t,l)=1}}
(p^f-1) 
\left(
\langle \tfrac{t}{l} - \tfrac{j}{q-1} \rangle
+
1 - \langle \tfrac{t}{l} + \tfrac{j}{q-1} \rangle
\right)
\\ \notag & \equiv
\sum_{\substack{1 \leq t < \frac{l}{2} \\ \gcd(t,l)=1}}
(p^f-1) 
\left(\tfrac{t}{l} - \tfrac{j}{q-1} -  \lfloor \tfrac{t}{l} - \tfrac{j}{q-1} \rfloor
 - \tfrac{t}{l} - \tfrac{j}{q-1} + \lfloor \tfrac{t}{l} + \tfrac{j}{q-1} \rfloor +1 \right)
 \\ \notag & \equiv
\sum_{\substack{1 \leq t < \frac{l}{2} \\ \gcd(t,l)=1}}
(p^f-1) 
\left( - \tfrac{2j}{q-1} -  \lfloor \tfrac{t}{l} - \tfrac{j}{q-1} \rfloor
 + \lfloor \tfrac{t}{l} + \tfrac{j}{q-1} \rfloor  + 1\right)
 \\ & \equiv 
0 \pmod 2, 
\end{align}
as $p^f-1$ is even and $(p^f-1) \tfrac{j}{q-1} = (q^{f^{\prime}}-1) \tfrac{j}{q-1} \in \mathbb{Z}$.
Combining (\ref{thm_GQ_for8})-(\ref{thm_GQ_for11})
we get that
\begin{equation}\label{thm_GQ_for12}
\prod_{\substack{h=1 \\ c_h \notin \mathbb{Z}}}^{\delta}
\prod_{k=0}^{r-1}
\frac{\biggfp{\langle \frac{c_h}{q-1} p^k \rangle}}{\biggfp{\langle (\frac{c_h}{q-1} -\frac{j}{q-1} )p^k \rangle}}
\frac{\biggfp{\langle -\frac{c_h}{q-1} p^k \rangle}}{\biggfp{\langle (-\frac{c_h}{q-1} +\frac{j}{q-1} )p^k \rangle}}
(-p)^{\lfloor{\langle \frac{c_h}{q-1} p^k \rangle -\frac{j p^k}{q-1}}\rfloor + \lfloor{\langle -\frac{c_h}{q-1} p^k\rangle +\frac{j p^k}{q-1}}\rfloor}
= 1.
\end{equation}
The product of (\ref{thm_GQ_for7}) and (\ref{thm_GQ_for12}) establishes (\ref{thm_GQ_for6}), which completes the proof.

\end{proof}

\begin{proof}[Proof of Theorem \ref{thm_G_Main}]
Let $(\{na_i\}, \{nb_i\})$ be defined over $\mathbb{Q}$ with exponents $(\{p_i : 1 \leq i \leq t \}, \{q_i  : 1 \leq i \leq s \})$.
Then, $(\{a_i+\frac{l}{n}\}, \{b_i+\frac{l}{n}\})$ are defined over $\mathbb{Q}$ with exponents $(\{n p_i  : 1 \leq i \leq t \}, \{n q_i : 1 \leq i \leq s\})$ as
\begin{align*}
\prod_{i=1}^{m} \prod_{l=0}^{n-1}  \frac{x- e^{2 \pi \I (a_i+\frac{l}{n})}}{x- e^{2 \pi \I (b_i +\frac{l}{n})}}
=
\prod_{i=1}^{m} \prod_{l=0}^{n-1}  \frac{x- e^{2 \pi \I a_i} e^{2 \pi \I \frac{l}{n}}}{x- e^{2 \pi \I b_i} e^{2 \pi \I \frac{l}{n}}}
= 
\prod_{i=1}^{m} \frac{x^n- e^{2 \pi \I n a_i}}{x^n- e^{2 \pi \I n b_i}}
=
\frac{\prod_{i=1}^{t} x^{np_i} - 1}{\prod_{i=1}^{s} x^{nq_i} - 1}.
\end{align*}
We note that
\begin{align*}
M_n:=\frac{\prod_{i=1}^{t}{(np_i)}^{np_i}}{\prod_{i=1}^{s}{(nq_i)}^{nq_i}}
=
n^{n(\sum_{i=1}^{t} p_i - \sum_{i=1}^{s} s_i)} \cdot \frac{\prod_{i=1}^{t}{p_i}^{np_i}}{\prod_{i=1}^{s}{q_i}^{nq_i}}
=
\left(\frac{\prod_{i=1}^{t}{p_i}^{p_i}}{\prod_{i=1}^{s}{q_i}^{q_i}}\right)^n
=
M^n.
\end{align*}
We also note that if $D(x) = \gcd({\prod_{i=1}^{t} x^{p_i} - 1},{\prod_{i=1}^{s} x^{q_i} - 1})$ has degree $\delta$, then
$D_n(x) := \gcd({\prod_{i=1}^{t} x^{np_i} - 1},{\prod_{i=1}^{s} x^{nq_i} - 1})$ has degree $n \delta$.
If $z=\exp\left(\frac{2 \pi \I \, c}{q-1}\right)$ is a zero of $D_n(x)$ with multiplicity $\mathfrak{s}_n(c)$ then $z^n=\exp\left(\frac{2 \pi \I \, n c}{q-1}\right)$ is a zero of $D(x)$ with the same multiplicity, i.e. $\mathfrak{s}(nc)=\mathfrak{s}_n(c)$.

As $(\{na_i\}, \{nb_i\})$ are defined over $\mathbb{Q}$ we use Theorem \ref{thm_GQ} expand the summands and then apply Proposition \ref{prop_Orth}, with $\chi=\bar{\omega}^j$, to get
\begin{align*}
\sum_{l=0}^{n-1}
{_{m}G_{m}} &
\biggl[ \begin{array}{cccc} n a_1, & n a_2, & \dotsc, & n a_m \\[2pt]
 n b_1, & n b_2, & \dotsc, & n b_m \end{array}
\Big| \; \zeta_n^l \cdot \lambda \; \biggr]_q\\
%&=
%\sum_{l=0}^{n-1}
%\frac{-1}{q-1}  \sum_{j=0}^{q-2} 
%(-1)^{j(m+\delta)}\;
%q^{-\mathfrak{s}(0)+\mathfrak{s}(j)}\;
%\bar{\omega}^j(M \cdot \zeta_n^l \cdot \lambda)\\
%& \qquad \times
%\prod_{k=0}^{r-1} 
%\prod_{i=1}^{t}  
%\biggfp{\langle \tfrac{-j p_i}{q-1} p^k \rangle}
%(-p)^{\langle \tfrac{-j p_i}{q-1} p^k \rangle}
%\prod_{i=1}^{s}  
%\biggfp{\langle \tfrac{j q_i}{q-1} p^k \rangle}
%(-p)^{\langle \tfrac{j q_i}{q-1} p^k \rangle}\\
&=
\frac{-1}{q-1}  \sum_{j=0}^{q-2} 
(-1)^{j(m+\delta)}\;
q^{-\mathfrak{s}(0)+\mathfrak{s}(j)}\;
\bar{\omega}^j(M \cdot \lambda)\\
& \qquad \times
\prod_{k=0}^{r-1} 
\prod_{i=1}^{t}  
\biggfp{\langle \tfrac{-j p_i}{q-1} p^k \rangle}
(-p)^{\lfloor \tfrac{-j p_i}{q-1} p^k \rfloor}
\prod_{i=1}^{s}  
\biggfp{\langle \tfrac{j q_i}{q-1} p^k \rangle}
(-p)^{\lfloor \tfrac{j q_i}{q-1} p^k \rfloor}\\
&\qquad \times
\sum_{l=0}^{n-1} \bar{\omega}^j(\zeta_n^l)
\\
&=
\frac{-n}{q-1}  \sum_{\substack{j=0 \\ j\equiv 0 \imod n}}^{q-2} 
(-1)^{j(m+\delta)}\;
q^{-\mathfrak{s}(0)+\mathfrak{s}(j)}\;
\bar{\omega}^j(M \cdot \lambda)\\
& \qquad \times
\prod_{k=0}^{r-1} 
\prod_{i=1}^{t}  
\biggfp{\langle \tfrac{-j p_i}{q-1} p^k \rangle}
(-p)^{\lfloor \tfrac{-j p_i}{q-1} p^k \rfloor}
\prod_{i=1}^{s}  
\biggfp{\langle \tfrac{j q_i}{q-1} p^k \rangle}
(-p)^{\lfloor \tfrac{j q_i}{q-1} p^k \rfloor}\\
&=
\frac{-n}{q-1}  \sum_{j=0}^{\frac{q-1}{n}-1} 
(-1)^{jn(m+\delta)}\;
q^{-\mathfrak{s}(0)+\mathfrak{s}(jn)}\;
\bar{\omega}^{jn}(M \cdot \lambda)\\
& \qquad \times
\prod_{k=0}^{r-1} 
\prod_{i=1}^{t}  
\biggfp{\langle \tfrac{-j n p_i}{q-1} p^k \rangle}
(-p)^{\lfloor \tfrac{-j n p_i}{q-1} p^k \rfloor}
\prod_{i=1}^{s}  
\biggfp{\langle \tfrac{j n q_i}{q-1} p^k \rangle}
(-p)^{\lfloor \tfrac{j n q_i}{q-1} p^k \rfloor}.
\end{align*}
If we denote the summand of the last expression above as $f(j)$ then it is easy to show that $f(j+b(\frac{q-1}{n}))=f(j)$ for all $0\leq b \leq n-1$.
Thus $\sum_{j=0}^{\frac{q-1}{n}-1} f(j) = \frac{1}{n} \sum_{j=0}^{q-2} f(j)$. So,
\begin{align*}
\sum_{l=0}^{n-1}
{_{m}G_{m}} &
\biggl[ \begin{array}{cccc} n a_1, & n a_2, & \dotsc, & n a_m \\[2pt]
 n b_1, & n b_2, & \dotsc, & n b_m \end{array}
\Big| \; \zeta_n^l \cdot \lambda \; \biggr]_q\\
&=
\frac{-1}{q-1}  \sum_{j=0}^{q-2} 
(-1)^{jn(m+\delta)}\;
q^{-\mathfrak{s}(0)+\mathfrak{s}(jn)}\;
\bar{\omega}^{jn}(M \cdot \lambda)\\
& \qquad \times
\prod_{k=0}^{r-1} 
\prod_{i=1}^{t}  
\biggfp{\langle \tfrac{-j n p_i}{q-1} p^k \rangle}
(-p)^{\lfloor \tfrac{-j n p_i}{q-1} p^k \rfloor}
\prod_{i=1}^{s}  
\biggfp{\langle \tfrac{j n q_i}{q-1} p^k \rangle}
(-p)^{\lfloor \tfrac{j n q_i}{q-1} p^k \rfloor}\\
&=
\frac{-1}{q-1}  \sum_{j=0}^{q-2} 
(-1)^{j(nm+n\delta)}\;
q^{-\mathfrak{s}_n(0)+\mathfrak{s}_n(j)}\;
\bar{\omega}^{j}(M_n \cdot \lambda^n)\\
& \qquad \times
\prod_{k=0}^{r-1} 
\prod_{i=1}^{t}  
\biggfp{\langle \tfrac{-j n p_i}{q-1} p^k \rangle}
(-p)^{\lfloor \tfrac{-j n p_i}{q-1} p^k \rfloor}
\prod_{i=1}^{s}  
\biggfp{\langle \tfrac{j n q_i}{q-1} p^k \rangle}
(-p)^{\lfloor \tfrac{j n q_i}{q-1} p^k \rfloor}\\
&=
{_{nm}G_{nm}}
\biggl[ \begin{array}{c} 
a_i+\frac{l}{n} : 1 \leq i \leq m, 0 \leq l \leq n-1 \\[2pt]
b_i+\frac{l}{n} : 1 \leq i \leq m, 0 \leq l \leq n-1 \end{array}
\Big| \; \lambda^n \; \biggr]_q.
\end{align*}
\end{proof}

\begin{proof}[Proof of Theorem \ref{thm_G_Main_Converse}]
This proof is similar to its finite field counterpart, Theorem \ref{thm_F_Main_Converse}.
We expand
${_{nm}G_{nm}} [\{a_i+\frac{l}{n}\};\{b_i+\frac{l}{n}\} \mid \lambda \;]_q$, by Definition \ref{def_Gq},
and make the change of variable $ j \rightarrow j - \frac{q-1}{n}$ to get that
\begin{equation*}
{_{nm}G_{nm}} [\{a_i+\tfrac{l}{n}\};\{b_i+\tfrac{l}{n}\} \mid \lambda \;]_q
=
\omega^{\frac{q-1}{n}}(\lambda)
\cdot
{_{nm}G_{nm}} [\{a_i+\tfrac{l}{n}\};\{b_i+\tfrac{l}{n}\} \mid \lambda \;]_q.
\end{equation*}
Noting that $\omega^{\frac{q-1}{n}}(\lambda) \neq 1$, as $\lambda$ is not an $n$-th power, completes the proof.
\end{proof}

%%%%%%%%%%%%%%%%%%%%%%%%%%%%%%%%%%%%%%%%%%%%%%%%%%
%%%%%%%%%%%%%%%%%%%%%%%%%%%%%%%%%%%%%%%%%%%%%%%%%%
%%%%%%%%%%%%%%%%%%%%%%%%%%%%%%%%%%%%%%%%%%%%%%%%%%
\subsection{Classical Setting}\label{subsec_Proofs_Classical}

\begin{proof}[Proof of Theorem \ref{thm_C_Main}]
By definition, we see that
$\ph{a}{nk} = n^{nk} \prod_{l=0}^{n-1} \ph{\tfrac{a}{n} + \tfrac{l}{n}}{k}.$
%\begin{equation*}
%\ph{a}{nk} = n^{nk} \prod_{l=0}^{n-1} \ph{\tfrac{a}{n} + \tfrac{l}{n}}{k}.
%\end{equation*}
Therefore,
\begin{align*}
\sum_{l=0}^{n-1}
{_mF_{m}} 
\biggl[ \begin{array}{cccc} 
n a_1, & n a_2, & \dotsc, & n a_m \vspace{.05in}\\
n b_1 & n b_2, & \dotsc, & n b_m \end{array}
\Big| \; \xi_n^l \cdot z \biggr]
&=
\sum_{l=0}^{n-1}
\sum^{\infty}_{k=0}
\prod_{i=1}^{m} 
\frac{\ph{na_i}{k}}{\ph{nb_i}{k}}
\; {(\xi_n^l z)^k}\\
&=
\sum^{\infty}_{k=0}
\prod_{i=1}^{m} 
\frac{\ph{na_i}{k}}{\ph{nb_i}{k}}
\; {z^k} \;
\sum_{l=0}^{n-1}
{\xi_n^{lk}}\\
&=
n
\sum^{\infty}_{\substack{k=0\\ k \equiv 0 \imod n}}
\prod_{i=1}^{m} 
\frac{\ph{na_i}{k}}{\ph{nb_i}{k}}
\; {z^k} \qquad (\textup{by Cor. } \ref{cor_Orth})\\
&=
n
\sum^{\infty}_{k=0}
\prod_{i=1}^{m} 
\frac{\ph{na_i}{nk}}{\ph{nb_i}{nk}}
\; {z^{nk}}\\
&=
n
\sum^{\infty}_{k=0}
\prod_{i=1}^{m} 
\prod_{l=0}^{n-1} 
\frac{\ph{a_i + \frac{l}{n}}{k}}{\ph{b_i + \frac{l}{n}}{k}}
\; {z^{nk}},
\end{align*}
as required.
\end{proof}

%%%%%%%%%%%%%%%%%%%%%%%%%%%%%%%%%%%%%%%%%%%%%%%%%%
%%%%%%%%%%%%%%%%%%%%%%%%%%%%%%%%%%%%%%%%%%%%%%%%%%
%%%%%%%%%%%%%%%%%%%%%%%%%%%%%%%%%%%%%%%%%%%%%%%%%%
%%%%%%%%%%%%%%%%%%%%%%%%%%%%%%%%%%%%%%%%%%%%%%%%%%
%%%%%%%%%%%%%%%%%%%%%%%%%%%%%%%%%%%%%%%%%%%%%%%%%%
%%%%%%%%%%%%%%%%%%%%%%%%%%%%%%%%%%%%%%%%%%%%%%%%%%

\section{Applications}\label{sec_Apps}
In this section we give some applications of our main results. 
The purpose is mainly to show the different types of applications, giving one or two illustrative examples in each case.
So, the list of such results shown here is by no means exhaustive. 

\subsection{Finite Field Setting}\label{subsec_Apps_FF}
Using Corollary \ref{cor_F_n=2} we can leverage known reduction formulas at lower orders to get new reduction formulas.  
We give examples when $m=2,3,4$.

\begin{cor}\label{cor_FF_n=2_m=2}
For $q$ odd and $A^2, B^4 \neq \varepsilon$,
\begin{equation*}
{_{4}F_{4}} {\biggl( \begin{array}{cccc} 
A, & \varphi A, & B, & \varphi B\\[1pt]
\varepsilon, & \varphi, & A\bar{B}, & \varphi A\bar{B}
 \end{array}
\Big| \; 1 \biggr)}_{q}
=
\frac{g(B^2)g(\bar{A^2} B^4)}{g(\bar{A^2} B^2)g(B^4)}
+
\sum_{R^2=A^2} \frac{g(\bar{A^2})g(\bar{R} B^2)}{g(\bar{R})g(\bar{A^2} B^2)}.
\end{equation*}
\end{cor}

\begin{proof}
Applying Corollary \ref{cor_F_n=2} with $m=2$, $\lambda=1$, $A_1=A$, $A_2=B$, $B_1=\varepsilon$ and $B_2=A\bar{B}$ we get that
\begin{multline*}
{_{4}F_{4}} {\biggl( \begin{array}{cccc} 
A, & \varphi A, & B, & \varphi B\\[1pt]
\varepsilon, & \varphi, & A\bar{B}, & \varphi A\bar{B}
 \end{array}
\Big| \; 1 \biggr)}_{q}\\
=
{_{2}F_{2}} {\biggl( \begin{array}{cc} 
A^2, & B^2\\[1pt]
\varepsilon, & A^2\bar{B^2} \end{array}
\Big| \; 1 \biggr)}_{q}
+
{_{2}F_{2}} {\biggl( \begin{array}{cc} 
A^2, & B^2\\[1pt]
\varepsilon, & A^2\bar{B^2} \end{array}
\Big| -1 \biggr)}_{q}.
\end{multline*}
Each ${_{2}F_{2}}$ can be reduced to an expression in terms of Gauss sums by Theorems 1.9 and 1.10 in \cite{McC6}, respectively, giving the desired result.
\end{proof}

\begin{cor}\label{cor_FF_n=2_m=3}
For $q \equiv 1 \imod 4$ and $A^8 \neq \varepsilon$,
\begin{multline*}
{_{6}F_{6}} {\biggl( \begin{array}{cccccc} 
\chi_4, & \bar{\chi_4}, & \chi_4 A, & \bar{\chi_4} A, & \chi_4 \bar{A}, & \bar{\chi_4 A}\\[1pt]
\varepsilon, & \varphi, & \bar{A}, & \varphi \bar{A}, & A, & \varphi A \end{array}
\Big| \; 1 \biggr)}_{q}\\
=
1
+\sum_{R^2=\varphi} \frac{ g(\bar{R}\varphi A^2) \, g(\bar{R}\varphi \bar{A^2}) }{g(\bar{R})^2} 
+
\begin{cases}
\displaystyle\sum_{S^2=\chi_4} \frac{ g(SA) \, g(\chi_4 SA)}{g(\varphi SA)  \, g(\bar{\chi_4} SA)} & \textup{if } q \equiv 1 \, (8),\\[3pt]
0 & \textup{otherwise}.
\end{cases}
\end{multline*}
\end{cor}

\begin{proof}
From \cite[(6.38)]{G} we can derive that, for $C^4 \neq \varepsilon$,
\begin{multline}\label{cor_FF_n=2_m=3_for1}
{_{3}F_{3}} {\biggl( \begin{array}{ccc} 
\varphi, & \varphi C, & \varphi \bar{C}\\[1pt]
\varepsilon, & \bar{C}, & C
 \end{array}
\Big| -1 \biggr)}_{q}\\
=
C(-1) \times 
\left[ 
1 + 
\begin{cases}
0 & \textup{if $\chi_4 C$ is not a square},\\
\displaystyle\sum_{D^2=\chi_4 C} \frac{ g(D) \, g(\chi_4 D)}{g(\varphi D)  \, g(\bar{\chi_4} D)} & \textup{otherwise}.
\end{cases}
\right].
\end{multline}
We apply Corollary \ref{cor_F_n=2} with $m=3$, $\lambda=1$, $A_1=\chi_4$, $A_2=\chi_4 A$, $A_3=\chi_4 \bar{A}$, $B_1=\varepsilon$, $B_2=\bar{A}$ and $B_3=A$ to split the ${_{6}F_{6}}$ in the statement of the corollary into ${_{3}F_{3}}(\cdots\mid+1)+{_{3}F_{3}}(\cdots\mid-1)$. The ${_{3}F_{3}}(\cdots\mid+1)$ can be reduced by \cite[Thm.~1.11]{McC6} and the ${_{3}F_{3}}(\cdots\mid-1)$ by (\ref{cor_FF_n=2_m=3_for1}), giving the desired result.
\end{proof}
The case corresponding to Corollary \ref{cor_FF_n=2_m=3} with $A=\varepsilon$ is covered by Corollary \ref{cor_phi3}.

\begin{cor}\label{cor_FF_n=2_m=4}
For $q \equiv 1 \imod 4$, $A^4  \neq \varepsilon$, $B^8 \neq \varepsilon$ and $A^2 \neq \varphi B^4$,
\begin{multline*}
{_{8}F_{8}} {\biggl( \begin{array}{cccccccc} 
A, & \varphi A, &  \chi_4 A, & \bar{\chi_4} A, & B, & \varphi B, &  \chi_4 B, & \bar{\chi_4} B\\[1pt]
\varepsilon, & \varphi, & \chi_4, & \bar{\chi_4}, & A\bar{B}, & \varphi A\bar{B}, & \chi_4 A\bar{B}, &  \bar{\chi_4} A\bar{B} \end{array}
\Big| \; 1 \biggr)}_{q}\\
=
\frac{g(B^4)g(\bar{A^4} B^8)}{g(\bar{A^4} B^4)g(B^8)}
+
\sum_{R^2=A^4} \frac{g(\bar{A^4})g(\bar{R} B^4)}{g(\bar{R})g(\bar{A^4} B^4)}\\
+
\dfrac{g(\bar{A^2}) \, g(\bar{A^2}\varphi B^4) }{g(\bar{A^2}B^2) \, g( \bar{A^2}\varphi B^2)}
\displaystyle\sum_{R^2=A^2} 
{_{3}F_2} \biggl( \begin{array}{ccc} R \varphi \bar{A^2}, & B^2, & \varphi B^2 \vspace{.02in}\\
\phantom{R \varphi \bar{A^2}} & R, & \varphi \end{array}
\Big| \; 1 \biggr)_{q}.
\end{multline*}
\end{cor}

\begin{proof}
We apply Corollary \ref{cor_F_n=2} with $m=4$, $\lambda=1$, $A_1=A$, $A_2=\chi_4 A$, $A_3=B$, $A_4=\chi_4 B$, $B_1=\varepsilon$, $B_2=\chi_4$, $B_3=A \bar{B}$ and $B_4=\chi_4 A \bar{B}$ to split the ${_{8}F_{8}}$ in the statement of the corollary into ${_{4}F_{4}}(\cdots\mid+1)+{_{4}F_{4}}(\cdots\mid-1)$. The ${_{4}F_{4}}(\cdots\mid+1)$ can be reduced by Corollary \ref{cor_FF_n=2_m=2} and the ${_{4}F_{4}}(\cdots\mid-1)$ by \cite[Thm.~1.5]{McC6}, giving the desired result.
\end{proof}

%%%%%%%%%%%%%%%%%%%%%%%%%%%%%%%%%%%%%%%%%%%%%%%%%%
%%%%%%%%%%%%%%%%%%%%%%%%%%%%%%%%%%%%%%%%%%%%%%%%%%
%%%%%%%%%%%%%%%%%%%%%%%%%%%%%%%%%%%%%%%%%%%%%%%%%%

\subsection{$p$-adic Setting}\label{subsec_Apps_padic}

We begin by looking at Corollary \ref{cor_G_n=2} in the case that $a_i = \frac{1}{4}$, $b_i=1$, i.e.,
\begin{multline}\label{cor_G_n=2_phi}
{_{m}G_{m}}
\biggl[ \begin{array}{cccc}
\frac{1}{2} & \frac{1}{2} & \dotsc & \frac{1}{2} \\[2pt]
1 & 1 & \dotsc & 1 \end{array}
\Big| \; \lambda \; \biggr]_q
+
{_{m}G_{m}}
\biggl[ \begin{array}{cccc} 
\frac{1}{2} & \frac{1}{2} & \dotsc & \frac{1}{2} \\[2pt]
1 & 1 & \dotsc & 1 \end{array}
\Big| \; -\lambda \; \biggr]_q\\
=
{_{2m}G_{2m}}
\biggl[ \begin{array}{ccccccc} 
\frac{1}{4} & \frac{3}{4} & \frac{1}{4} & \frac{3}{4} & \dotsm & \frac{1}{4} & \frac{3}{4} \\[2pt]
1 & \frac{1}{2} & 1 & \frac{1}{2}  & \dotsm & 1 & \frac{1}{2} 
 \end{array}
\Big| \; \lambda^2 \; \biggr]_q.
\end{multline}
We will consider (\ref{cor_G_n=2_phi}) for $m=2,3, 4$ and various values of $\lambda$.

\begin{cor}[$m=2, \lambda=1$]\label{cor_phi2}
Let $q=p^r$ be an odd prime power. When $q \equiv 1 \imod {4}$, we write $q=x^2 +y^2$ for integers $x$ and $y$, such that $x\equiv 1 \imod{4}$, and $p \nmid x$ when $p \equiv 1 \imod{4}$. Then
\begin{equation*}
{_{4}G_{4}}
\biggl[ \begin{array}{cccc} 
\frac{1}{4} & \frac{3}{4} & \frac{1}{4} & \frac{3}{4} \\[2pt]
1 & \frac{1}{2} & 1 & \frac{1}{2}
\end{array}
\Big| \; 1 \; \biggr]_q
=
\begin{cases}
-1 & \textup{if } q\equiv 3 \imod 4,\\
1+2x & \textup{if } q\equiv 1 \imod 8,\\ 
1-2x & \textup{if } q\equiv 5 \imod 8.
\end{cases}
\end{equation*}
\end{cor}

\begin{proof}%[Proof of Corollary \ref{cor_phi2}]
By \cite[Thm.~4.9]{G2},
\begin{equation}\label{cor_phi2_for1} 
{_{2}G_{2}}
\biggl[ \begin{array}{cccc}
\frac{1}{2} & \frac{1}{2} \\[2pt]
1 & 1 \end{array}
\Big| \; 1 \; \biggr]_q
=
\begin{cases}
-1 & \textup{if } q \equiv 3 \imod 4,\\
+1 & \textup{if } q \equiv 1 \imod 4.\\
\end{cases}
\end{equation}
Combining \cite[Thm~1.10]{McC6} and \cite[Lemma~3.5]{DM} we have
\begin{equation}\label{cor_phi2_for2} 
{_{2}G_{2}}
\biggl[ \begin{array}{cccc}
\frac{1}{2} & \frac{1}{2} \\[2pt]
1 & 1 \end{array}
\Big| -1 \; \biggr]_q
=
\begin{cases}
0 & \textup{if } q\equiv 3 \imod 4,\\
+2x & \textup{if } q\equiv 1 \imod 8,\\ 
-2x & \textup{if } q\equiv 5 \imod 8.
\end{cases}
\end{equation}
Substituting (\ref{cor_phi2_for1}) and (\ref{cor_phi2_for2}) into (\ref{cor_G_n=2_phi}), when $m=2$ and $\lambda=1$, yields the result.
\end{proof}

\begin{cor}[$m=3, \lambda=1$]\label{cor_phi3}
Let $q=p^r$ be an odd prime power. 
When $q \equiv 1 \imod {4}$, we write $q=x^2 +y^2$ for integers $x$ and $y$, such that $p \nmid x$ when $p \equiv 1 \imod{4}$. 
When $q \equiv 1,3 \imod {8}$, we write $q=u^2+2v^2$ for integers $u$ and $v$, such that $p \nmid u$ when $p \equiv 1,3 \imod{8}$.
Then
\begin{equation*}
{_{6}G_{6}}
\biggl[ \begin{array}{cccccc} 
\frac{1}{4} & \frac{3}{4} & \frac{1}{4} & \frac{3}{4} & \frac{1}{4} & \frac{3}{4} \\[2pt]
1 & \frac{1}{2} & 1 & \frac{1}{2} & 1 & \frac{1}{2}
\end{array}
\Big| \; 1 \; \biggr]_q
=
\begin{cases}
4(x^2+u^2)-3q & \textup{if } q\equiv 1 \imod 8,\\
q - 4u^2 & \textup{if } q\equiv 3 \imod 8,\\ 
4x^2-q & \textup{if } q\equiv 5 \imod 8,\\
-q & \textup{if } q\equiv 7 \imod 8.
\end{cases}
\end{equation*}
\end{cor}

\begin{proof}%[Proof of Corollary \ref{cor_phi3}]
Combining \cite[Thm~1.11]{McC6} and \cite[Lemma~3.5]{DM} we have
\begin{equation}\label{cor_phi3_for1} 
{_{3}G_{3}}
\biggl[ \begin{array}{cccc}
\frac{1}{2} & \frac{1}{2} & \frac{1}{2} \\[2pt]
1 & 1 & 1 \end{array}
\Big| \; 1 \; \biggr]_q
=
\begin{cases}
0 & \textup{if } q \equiv 3 \imod 4,\\
4x^2-2q & \textup{if } q \equiv 1 \imod 4.\\
\end{cases}
\end{equation}
From \cite{EG} we have
\begin{equation}\label{cor_phi3_for2} 
{_{3}G_{3}}
\biggl[ \begin{array}{cccc}
\frac{1}{2} & \frac{1}{2} & \frac{1}{2} \\[2pt]
1 & 1 & 1 \end{array}
\Big| -1 \; \biggr]_q
=
\begin{cases}
4u^2-q & \textup{if } q\equiv 1 \imod 8,\\
-4u^2+q & \textup{if } q\equiv 3 \imod 8,\\ 
q & \textup{if } q\equiv 5 \imod 8,\\
-q & \textup{if } q\equiv 7 \imod 8.
\end{cases}
\end{equation}
Substituting (\ref{cor_phi3_for1}) and (\ref{cor_phi3_for2}) into (\ref{cor_G_n=2_phi}), when $m=3$ and $\lambda=1$, yields the result.
\end{proof}

When $q=p$ we can relate the hypergeometric functions in (\ref{cor_phi2_for1}), (\ref{cor_phi3_for1}) and (\ref{cor_phi3_for2}) to the $p$-th Fourier coefficients of certain modular forms.
All modular forms will be denoted $f_{\#}$, where the subscript $\#$ is the form's unique identifier from LMFDB \cite{LMFDB}, and will have Fourier expansion 
$f_{\#} = \sum_{n\geq0} a_n(f_{\#}) q^n$.

\begin{cor}\label{cor_phi2_mod}
Consider the modular form $f_{32.2.a.a} = \eta^2(4z) \eta^2(8z) = \sum_{n\geq1} a_n(f_{32.2.a.a}) q^n \in S_2(\Gamma_0(32))$. 
If $p$ is an odd prime, then 
\begin{equation*}
{_{4}G_{4}}
\biggl[ \begin{array}{cccc} 
\frac{1}{4} & \frac{3}{4} & \frac{1}{4} & \frac{3}{4} \\[2pt]
1 & \frac{1}{2} & 1 & \frac{1}{2}
\end{array}
\Big| \; 1 \; \biggr]_p
=
\varphi(-1) + a_p(f_{32.2.a.a}).
\end{equation*}
\end{cor}

\begin{proof}
From \cite[Prop.~1 \& Thm. 2]{O} we have
\begin{equation}\label{cor_phi2_mod_for2} 
{_{2}G_{2}}
\biggl[ \begin{array}{cccc}
\frac{1}{2} & \frac{1}{2} \\[2pt]
1 & 1 \end{array}
\Big| -1 \; \biggr]_p
=
a_p(f_{32.2.a.a}).
\end{equation}
Substituting (\ref{cor_phi2_for1}) and (\ref{cor_phi2_mod_for2})  into (\ref{cor_G_n=2_phi}), when $m=2$ and $\lambda=1$, yields the result.
\end{proof}

\begin{cor}\label{cor_phi3_mod}
Consider the modular forms 
$f_{16.3.c.a}=\eta^6(4z) =\sum_{n\geq1} a_n(f_{16.3.c.a}) q^n \in S_3(\Gamma_0(16),(\tfrac{-4}{\cdot}))$ 
and 
$f_{256.2.a.a}=\sum_{n\geq1} a_n(f_{256.2.a.a}) q^n \in S_2(\Gamma_0(256))$. 
If $p$ is an odd prime, then 
\begin{equation*}
{_{6}G_{6}}
\biggl[ \begin{array}{cccccc} 
\frac{1}{4} & \frac{3}{4} & \frac{1}{4} & \frac{3}{4} & \frac{1}{4} & \frac{3}{4} \\[2pt]
1 & \frac{1}{2} & 1 & \frac{1}{2} & 1 & \frac{1}{2}
\end{array}
\Big| \; 1 \; \biggr]_p
=
a_p(f_{16.3.c.a})+ \varphi(2) \cdot a_{p^2}(f_{256.2.a.a}).
\end{equation*}
\end{cor}

\begin{proof}
The relation
\begin{equation}\label{cor_phi3_mod_for1} 
{_{3}G_{3}}
\biggl[ \begin{array}{cccc}
\frac{1}{2} & \frac{1}{2} & \frac{1}{2} \\[2pt]
1 & 1 & 1 \end{array}
\Big| \; 1 \; \biggr]_p
=
a_p(f_{16.3.c.a})
\end{equation}
corresponds to one of Rodriguez Villegas supercongruence conjectures \cite{RV} and can be found in\cite{M}.
From \cite[Thms.~5 \& 6]{O} we have
\begin{equation}\label{cor_phi3_mod_for2} 
{_{3}G_{3}}
\biggl[ \begin{array}{cccc}
\frac{1}{2} & \frac{1}{2} & \frac{1}{2} \\[2pt]
1 & 1 & 1 \end{array}
\Big| -1 \; \biggr]_p
=
\varphi(2) (a_p(f_{256.2.a.a})^2-p)
=
\varphi(2) \cdot a_{p^2}(f_{256.2.a.a}).
\end{equation}
Substituting (\ref{cor_phi3_mod_for1}) and (\ref{cor_phi3_mod_for2}) into (\ref{cor_G_n=2_phi}), when $m=3$ and $\lambda=1$, yields the result.
\end{proof}

In fact, we can produce similar results to Corollary \ref{cor_phi3_mod}, for almost any rational $\lambda$, via the following result of Ono's combined with the modularity theorem.
\begin{theorem}[\cite{O} Thm.~5]\label{thm_Ono_3F2_EC}
Let $t \in \mathbb{Q}-\{0,4\}$ and consider the elliptic curve $E_t/\mathbb{Q}$ given by 
$$E_{t}: y^2= x^3 - t^2 x^2 + (4t^3-t^4) x + t^6-4t^5.$$
If $p$ is an odd prime for which  ${\rm ord}_p(t(t-4))= 0$, then
\begin{align*}
{_{3}G_{3}}
\biggl[ \begin{array}{cccc}
\frac{1}{2} & \frac{1}{2} & \frac{1}{2} \\[2pt]
1 & 1 & 1 \end{array}
\Big| \; \frac{4- t}{4} \; \biggr]_p
= 
\varphi(t^2-4t)\left(a_p(E_{t})^2-p\right),
\end{align*}
where $a_p(E_{t}) =  p+1 - \#E_t(\mathbb{F}_p).$
\end{theorem}

Here are some examples where one of the modular forms has complex multiplication (CM), corresponding to the elliptic curves in \cite[Thm.~6]{O}.
\begin{cor}\label{cor_phi3_mod_t}
Let 
$$
{_{6}G_{6}}
\bigl[\; \lambda \; \bigr]_p
:=
{_{6}G_{6}}
\biggl[ \begin{array}{cccccc} 
\frac{1}{4} & \frac{3}{4} & \frac{1}{4} & \frac{3}{4} & \frac{1}{4} & \frac{3}{4} \\[2pt]
1 & \frac{1}{2} & 1 & \frac{1}{2} & 1 & \frac{1}{2}
\end{array}
\Big| \; \lambda \; \biggr]_p.
$$
\begin{enumerate}
\item
For $p\neq 2,3$,
\begin{align*}
{_{6}G_{6}}
\bigl[ \; \tfrac{1}{64} \; \bigr]_p
&=
\varphi(-7) \cdot a_{p^2}(f_{1568.2.a.a}) + a_{p^2}(f_{32.2.a.a})\\
&=
\varphi(-7) \cdot a_{p^2}(f_{1568.2.a.a})
+
\begin{cases}
-p & \textup{if } p \equiv 3 \imod 4,\\
4x^2-p & \textup{if } p \equiv 1 \imod 4,\\
\end{cases}
\end{align*}
where  $p=x^2+y^2 \equiv 1 \imod 4$, with $x$ odd.

\item
For $p\neq 2,3$,
\begin{align*}
{_{6}G_{6}}
\bigl[ \; 64 \; \bigr]_p
&=
\varphi(14) \cdot a_{p^2}(f_{1568.2.a.a}) + \varphi(2) \cdot a_{p^2}(f_{32.2.a.a})\\
&=
\varphi(14) \cdot a_{p^2}(f_{1568.2.a.a})
+\varphi(2)
\begin{cases}
- p & \textup{if } p \equiv 3 \imod 4,\\
4x^2-p & \textup{if } p \equiv 1 \imod 4,\\
\end{cases}
\end{align*}
where  $p=x^2+y^2 \equiv 1 \imod 4$, with $x$ odd.

\item
For $p\neq 2,5$,
\begin{align*}
{_{6}G_{6}}
\bigl[ \; \tfrac{1}{16} \; \bigr]_p
&=
\varphi(5) \cdot a_{p^2}(f_{200.2.a.b}) + \varphi(-3) \cdot a_{p^2}(f_{36.2.a.a})\\
&=
\varphi(5) \cdot a_{p^2}(f_{200.2.a.b})
+\varphi(-3)
\begin{cases}
- p & \textup{if } p \equiv 2 \imod 3,\\
4x^2-p & \textup{if } p \equiv 1 \imod 3,\\
\end{cases}
\end{align*}
where  $p=x^2+3y^2 \equiv 1 \imod 3$.

\item
For $p\neq 2,5$,
\begin{align*}
{_{6}G_{6}}
\bigl[ \; 16 \; \bigr]_p
&=
\varphi(5) \cdot a_{p^2}(f_{200.2.a.b}) + \varphi(3) \cdot a_{p^2}(f_{36.2.a.a})\\
&=
\varphi(5) \cdot a_{p^2}(f_{200.2.a.b})
+\varphi(3)
\begin{cases}
- p & \textup{if } p \equiv 2 \imod 3,\\
4x^2-p & \textup{if } p \equiv 1 \imod 3,\\
\end{cases}
\end{align*}
where  $p=x^2+3y^2 \equiv 1 \imod 3$.

\item
For $p\neq 2,3,5,13$,
\begin{align*}
{_{6}G_{6}}
\bigl[ \; \tfrac{1}{4096} \; \bigr]_p
&=
\varphi(65) \cdot a_{p^2}(f_{4225.2.a.h}) + \varphi(-7) \cdot a_{p^2}(f_{49.2.a.a})\\
&=
\varphi(65) \cdot a_{p^2}(f_{4225.2.a.h})
+\varphi(-7)
\begin{cases}
- p & \textup{if } p \equiv 3,5,6 \imod 7,\\
4x^2-p & \textup{if } p \equiv 1,2,4 \imod 7,\\
\end{cases}
\end{align*}
where  $p=x^2+7y^2 \equiv 1,2,4 \imod 7$.

\item
For $p\neq 2,3,5,13$,
\begin{align*}
{_{6}G_{6}}
\bigl[ \; 4096 \; \bigr]_p
&=
\varphi(65) \cdot a_{p^2}(f_{4225.2.a.h}) + \varphi(7) \cdot a_{p^2}(f_{49.2.a.a})\\
&=
\varphi(65) \cdot a_{p^2}(f_{4225.2.a.h})
+\varphi(7)
\begin{cases}
- p & \textup{if } p \equiv 3,5,6 \imod 7,\\
4x^2-p & \textup{if } p \equiv 1,2,4 \imod 7,\\
\end{cases}
\end{align*}
where  $p=x^2+7y^2 \equiv 1,2,4 \imod 7$.
\end{enumerate}
\end{cor}

\begin{proof}
We use (\ref{cor_G_n=2_phi}) with $m=3$. We then consider Theorem \ref{thm_Ono_3F2_EC} for appropriate values of $t$ and combine with the modularity theorem (curves and modular forms are determined in Sage \cite{SAGE} and LMFDB \cite{LMFDB}). We follow \cite[Thm.~6]{O} for evaluation of coefficients in the CM forms.
(1) $t=\frac{7}{2}, \frac{9}{2}$.
(3) $t=5,3$.
(5) $t= \frac{65}{16}, \frac{63}{16}$.
For cases (2), (4) and (6) we note that \cite[Thm.~4.2]{G2}
$$
{_{3}G_{3}}
\biggl[ \begin{array}{cccc}
\frac{1}{2} & \frac{1}{2} & \frac{1}{2} \\[2pt]
1 & 1 & 1 \end{array}
\Big| \; \lambda \; \biggr]_p
=
\varphi(-\lambda) \cdot
{_{3}G_{3}}
\biggl[ \begin{array}{cccc}
\frac{1}{2} & \frac{1}{2} & \frac{1}{2} \\[2pt]
1 & 1 & 1 \end{array}
\Big| \; \frac{1}{\lambda} \; \biggr]_p.
$$
\end{proof}

We can perform a similar exercise in the case of $m=2$ for $a_1=\frac{1}{8}$, $a_2=\frac{3}{8}$, $b_1=\frac{1}{6}$ and $b_2=\frac{1}{3}$, using the following result of the first author.
\begin{theorem}[\cite{McC7} Thm. 1.2]\label{thm_trace}
Let $p>3$ be prime. Consider an elliptic curve $E_{a,b}/\mathbb{F}_p$ of the form $E_{a,b}:y^2 = x^3+ax+b$, with $j(E_{a,b})\neq 0, 1728$. 
Define 
$a_p(E_{a,b}) := p+1-\# E_{a,b}(\mathbb{F}_p)$.
Then
\begin{equation}\label{for_main}
a_p(E_{a,b}) = \varphi(b) \cdot p \cdot
{_{2}G_{2}}
\biggl[ \begin{array}{cc} 
\frac{1}{4} & \frac{3}{4}\\[2pt]
\frac{1}{3}  & \frac{2}{3} \end{array}
\Big| \; {-\frac{27b^2}{4a^3}} \; \biggr]_p. 
\end{equation}
\end{theorem}

Combining Theorem \ref{thm_trace} and Corollary \ref{cor_G_n=2}, with $m=2$, the following corollary is immediate.
\begin{cor}\label{cor_Trace}
Let $p>3$ be prime.
For $E_{a,b}/\mathbb{F}_p$ and $E_{-a,b}/\mathbb{F}_p$ elliptic curves with $j(E_{\pm a,b})\neq 0, 1728$,
\begin{align*}
{_{4}G_{4}}
\biggl[ \begin{array}{cccc} 
\frac{1}{8} & \frac{5}{8} & \frac{3}{8} & \frac{7}{8} \\[2pt]
\frac{1}{6} & \frac{2}{3} & \frac{1}{3} & \frac{5}{6}
\end{array}
\Big| \;  {\frac{3^6 b^4}{2^4a^6}} \; \biggr]_p
=
\frac{\varphi(b)}{p}
\bigl(
a_p(E_{a,b}) + a_p(E_{-a,b})
\bigr).
\end{align*}
\end{cor}

By taking different values of $a$ and $b$ we can relate special values of the ${_{4}G_{4}}$ in Corollary \ref{cor_Trace} to coefficients of modular forms. Here is one such example with $a=1$ and $b=1$. Again, we use Sage \cite{SAGE} and LMFDB \cite{LMFDB} to identify the relevant elliptic curves and modular forms. 
\begin{cor}\label{cor_Trace_eg}
For $p>3$,
\begin{equation*}
p \cdot{_{4}G_{4}}
\biggl[ \begin{array}{cccc} 
\frac{1}{8} & \frac{5}{8} & \frac{3}{8} & \frac{7}{8} \\[2pt]
\frac{1}{6} & \frac{2}{3} & \frac{1}{3} & \frac{5}{6}
\end{array}
\Big| \;  \frac{3^6}{2^4} \; \biggr]_p
=
\varphi(-1) \cdot a_{p}(f_{248.2.a.c}) +   a_{p}(f_{92.2.a.a}).
\end{equation*}
\end{cor}

Barman and Saikia \cite{BS} provide similar results to Theorem \ref{thm_trace} for the elliptic curves $E_1: y^2 = x^3+ax^2+b$ and $E_2: y^2 = x^3+ax^2+bx$. Their results can also be used to provide relations similar to those in Corollaries \ref{cor_Trace} and \ref{cor_Trace_eg}.

We now examine an application of (\ref{cor_G_n=2_phi}) in the case $m=4$.
\begin{cor}\label{cor_G_n=2_phi_m=4}
For $p$ an odd prime,
\begin{equation*}
{_{8}G_{8}}
\biggl[ \begin{array}{cccccccc} 
\frac{1}{4} & \frac{3}{4} & \frac{1}{4} & \frac{3}{4} & \frac{1}{4} & \frac{3}{4} & \frac{1}{4} & \frac{3}{4} \\[2pt]
1 & \frac{1}{2} & 1 & \frac{1}{2} & 1 & \frac{1}{2} & 1 & \frac{1}{2}
\end{array}
\Big| \; 1 \; \biggr]_p
=
a_{p}(f_{8.4.a.a})+a_{p}(f_{32.2.a.a}) \cdot a_{p}(f_{32.3.c.a})+p
\end{equation*}
Furthermore, the right-hand side reduces to $a_{p}(f_{8.4.a.a})+p$ when $p \equiv 3 \imod 4$.
\end{cor}

\begin{proof}
By a result of Ahlgren and Ono \cite[Thm.~6]{AO} we have, for $p$ odd,
\begin{equation}\label{cor_G_n=2_phi_m=4_for1}
{_{4}G_{4}}
\biggl[ \begin{array}{cccc}
\frac{1}{2} & \frac{1}{2} & \frac{1}{2} & \frac{1}{2} \\[2pt]
1 & 1 & 1 & 1 \end{array}
\Big| \; 1 \; \biggr]_p
=
a_{p}(f_{8.4.a.a})+p.
\end{equation}
In \cite{MP}, the first author and Papanikolas proved that, for $p$ odd,
\begin{equation}\label{cor_G_n=2_phi_m=4_for2}
{_{4}G_{4}}
\biggl[ \begin{array}{cccc}
\frac{1}{2} & \frac{1}{2} & \frac{1}{2} & \frac{1}{2} \\[2pt]
1 & 1 & 1 & 1 \end{array}
\Big| \; -1 \; \biggr]_p
=
a_{p}(f_{32.2.a.a}) \cdot a_{p}(f_{32.3.c.a}),
\end{equation}
and both sides of (\ref{cor_G_n=2_phi_m=4_for2}) are zero when $p \equiv 3 \imod 4$. 
Substituting (\ref{cor_G_n=2_phi_m=4_for1}) and (\ref{cor_G_n=2_phi_m=4_for2}) into (\ref{cor_G_n=2_phi}), when $m=4$ and $\lambda=1$, yields the result.
\end{proof}

%%%%%%%%%%%%%%%%%%%%%%%%%%%%%%%%%%%%%%%%%%%%%%%%%%
%%%%%%%%%%%%%%%%%%%%%%%%%%%%%%%%%%%%%%%%%%%%%%%%%%
%%%%%%%%%%%%%%%%%%%%%%%%%%%%%%%%%%%%%%%%%%%%%%%%%%

\subsection{Classical Setting}\label{subsec_Apps_Classical}
Similar to the development of Corollaries \ref{cor_FF_n=2_m=2}-\ref{cor_FF_n=2_m=4} in the finite field setting, 
we can use Theorem \ref{thm_C_Main}, with $n=2$, to leverage known reduction formulas at lower orders to produce reduction formulas at higher orders in the classical setting.

\begin{cor}\label{cor_C_n=2_m=2}
For $\textup{Re}(b) < \frac{1}{4}$,
\begin{multline*}
{_4F_{4}} 
\biggl[ \begin{array}{cccc} 
a, & a +\frac{1}{2}, & b, & b +\frac{1}{2} \vspace{.05in}\\
1, & \frac{1}{2}, & \frac{1}{2}+a-b, & 1+a-b \end{array}
\Big| \; 1 \biggr]\\
=
\frac{1}{2}
\times
\left[
\frac{\gf{1+2a-2b} \gf{1-4b}}{\gf{1-2b} \gf{1+2a-4b}}
+
\frac{\gf{1+2a-2b} \gf{1+a}}{\gf{1+2a} \gf{1+a-2b}}
\right].
\end{multline*}
\end{cor}

\begin{proof}
We use Theorem \ref{thm_C_Main} with $n=m=2$, $\lambda=1$, $a_1=a$, $a_2=b$, $b_1=\frac{1}{2}$ and $b_2=\frac{1}{2}+a-b$ 
to split the ${_4}F_{4}$ in the statement of the corollary into ${_{2}F_{2}}\left[\cdots\mid+1\right]+{_{2}F_{2}}\left[\cdots\mid-1\right]$. 
The ${_{2}F_{2}}\left[\cdots\mid+1\right]$ can be reduced by Gauss' theorem \cite[p.~2]{Ba} and the ${_{2}F_{2}}\left[\cdots\mid-1\right]$ by Kummer's theorem \cite[p.~9]{Ba}, giving the desired result.
\end{proof}

Corollary \ref{cor_C_n=2_m=2} is a direct analogue of Corollary \ref{cor_FF_n=2_m=2} and is already known. It appears (without proof) in \cite[eqn.~11, p.~555]{PBM}.

\begin{cor}\label{cor_C_n=2_m=3}
\begin{multline*}
{_6F_{6}} 
\biggl[ \begin{array}{cccccc} 
\frac{1}{4}, & \frac{3}{4}, & \frac{1}{4}+a, & \frac{3}{4} + a, & \frac{1}{4}-a, & \frac{3}{4} - a \vspace{.05in}\\
1, & \frac{1}{2}, & 1-a, & \frac{1}{2}-a, & 1+a, & \frac{1}{2}+a \end{array}
\Big| \; 1 \biggr]\\
=
\frac{\gf{1+2a} \gf{1-2a}}{4 \pi}
\left[
\frac{\gf{\frac{1}{4}}^2}{\gf{\frac{3}{4}-2a} \gf{\frac{3}{4}+2a}}
+
\frac{\sqrt{2} \, \pi^2}{\gf{a+\frac{5}{8}} \gf{a+\frac{7}{8}} \gf{-a+\frac{5}{8}} \gf{-a+\frac{7}{8}} }
\right].
\end{multline*}
\end{cor}

\begin{proof}
Combining Whipple's \cite[eqn.~(9.3)]{Wh} and \cite[eqn.~(10.3)]{Wh} we get that
\begin{equation}\label{cor_C_n=2_m=3_for1}
{_3F_{3}} 
\biggl[ \begin{array}{ccc} 
\frac{1}{2}, & \frac{1}{2}+a, & \frac{1}{2}-a \vspace{.05in}\\
1, & 1+a, & 1-a \end{array}
\Big| \; -1 \biggr]
=
\frac{\pi \, \gf{1+a} \gf{1-a}}{\sqrt{2} \,\gf{\frac{a}{2}+\frac{5}{8}} \gf{\frac{a}{2}+\frac{7}{8}} \gf{-\frac{a}{2}+\frac{5}{8}} \gf{-\frac{a}{2}+\frac{7}{8}}}.
\end{equation}
We now use Theorem \ref{thm_C_Main} with $n=2$, $m=3$, $\lambda=1$, $a_1=\frac{1}{4}$, $a_2=\frac{1}{4}+a$, $a_3=\frac{1}{4}-a$, $b_1=\frac{1}{2}$, $b_2=\frac{1}{2}-a$ and $b_3=\frac{1}{2}+a$ 
to split the ${_6}F_{6}$ in the statement of the corollary into ${_{3}F_{3}}\left[\cdots\mid+1\right]+{_{3}F_{3}}\left[\cdots\mid-1\right]$. 
The ${_{3}F_{3}}\left[\cdots\mid+1\right]$ can be evaluated using Dixon's theorem \cite[eqn.~(10.1)]{Wh} and the ${_{2}F_{2}}\left[\cdots\mid-1\right]$ by (\ref{cor_C_n=2_m=3_for1}), giving the desired result.
\end{proof}

Corollary \ref{cor_C_n=2_m=3} is a direct analogue of Corollary \ref{cor_FF_n=2_m=3}. We are conscious of the fact that it is unlikely such a result is not already known. However, we haven't found Corollary \ref{cor_C_n=2_m=3} in the literature.

\begin{cor}\label{cor_C_n=2_m=4}
For $\textup{Re}(b) < \frac{1}{8}$,
\begin{multline*}
{_8F_{8}} 
\biggl[ \begin{array}{cccccccc} 
a, & a +\frac{1}{2}, & a +\frac{1}{4}, & a +\frac{3}{4}, & b, & b +\frac{1}{2}, & b +\frac{1}{4}, & b +\frac{3}{4} \vspace{.05in}\\
1, & \frac{1}{2}, & \frac{1}{4}, & \frac{3}{4}, & \frac{1}{2}+a-b, & 1+a-b, & \frac{1}{4}+a-b, & \frac{3}{4}+a-b \end{array}
\Big| \; 1 \biggr]
\\[3pt]
=
\frac{1}{4}
\times
\left[
\frac{\gf{1+4a-4b} \gf{1-8b}}{\gf{1-4b} \gf{1+4a-8b}}
+
\frac{\gf{1+4a-4b} \gf{1+2a}}{\gf{1+4a} \gf{1+2a-4b}}
\right]
\\[3pt]
+
\frac{\gf{1+2a-2b} \, \gf{\frac{1}{2}+2a-2b}}{2 \, \gf{1+2a} \,\gf{\frac{1}{2}+2a-4b}} \;
{_{3}F_2} \biggl[ \begin{array}{ccc} 
\frac{1}{2}-a, &2b, & 2b +\frac{1}{2}\vspace{.05in}\\
\phantom{\frac{1}{2}-a} & 1+a, & \frac{1}{2} \end{array}
\Big| \; 1 \biggr].
\end{multline*}
\end{cor}

\begin{proof}
Again, we use Theorem \ref{thm_C_Main}, this time with $n=2$, $m=4$, $\lambda=1$, $a_1=a$, $a_2=a+\frac{1}{4}$, $a_3=b$, $a_4=b+\frac{1}{4}$, $b_1=\frac{1}{2}$, $b_2=\frac{1}{4}$, $b_3=\frac{1}{2}+a-b$ and $b_4=\frac{1}{4}+a-b$.
This splits the ${_8}F_{8}$ in the statement of the corollary into ${_{4}F_{4}}\left[\cdots\mid+1\right]+{_{4}F_{4}}\left[\cdots\mid-1\right]$. 
The ${_{4}F_{4}}\left[\cdots\mid+1\right]$ can be evaluated using Corollary \ref{cor_C_n=2_m=2} and the ${_{4}F_{4}}\left[\cdots\mid-1\right]$ by a result of Whipple \cite[eqn.~(3.4)]{Wh}, giving the desired result.
\end{proof}

Corollary \ref{cor_C_n=2_m=4} is a direct analogue of Corollary \ref{cor_FF_n=2_m=4} and, again, we haven't found it in the literature.

%%%%%%%%%%%%%%%%%%%%%%%%%%%%%%%%%%%%%%%%%%%%%%%%%%
%%%%%%%%%%%%%%%%%%%%%%%%%%%%%%%%%%%%%%%%%%%%%%%%%%
%%%%%%%%%%%%%%%%%%%%%%%%%%%%%%%%%%%%%%%%%%%%%%%%%%
%%%%%%%%%%%%%%%%%%%%%%%%%%%%%%%%%%%%%%%%%%%%%%%%%%
%%%%%%%%%%%%%%%%%%%%%%%%%%%%%%%%%%%%%%%%%%%%%%%%%%
%%%%%%%%%%%%%%%%%%%%%%%%%%%%%%%%%%%%%%%%%%%%%%%%%%

\section{More on the Relationship between $F(\cdots)_q$ and $G[\cdots]_q$}\label{sec_GF}
We recall the discussion at the end of Section \ref{subsec_Res_FF} and beginning of Section \ref{subsec_Res_padic} regarding the domains of $F(\cdots)_q$ and $G[\cdots]_q$, when considered as functions of $q$.
Let $\mathbb{F}_q^{*} = \langle T \rangle$.
For fixed $a_i, b_i \in \mathbb{Q}$, if we consider $_mF_m(\{T^{a_i (q-1)}\}; \{T^{b_i(q-1)}\} \mid \cdot \,)_q$ to be a function of $q$, then the domain is all $q \equiv 1 \imod d$, where $d$ is the least common denominator of the elements in $\{a_i\} \cup \{b_i\}$, ensuring $a_i (q-1), b_i (q-1) \in \mathbb{Z}$. 
Via Lemma \ref{lem_G_to_F}, $_mF_m(\{T^{a_i (q-1)}\}; \{T^{b_i(q-1)}\} \mid \cdot \,)_q$ extends to ${_{m}G_{m}}[\{a_i\}; \{b_i\} \mid \cdot \,]_q$, whose domain is all $q$ relatively prime to $d$.

It is possible to extend the domain of $_mF_m(\{T^{a_i (q-1)}\}; \{T^{b_i(q-1)}\} \mid \cdot \,)_q$ without moving to the $p$-adic setting, in certain circumstances. Specifically, if the parameters $(\{a_i\}, \{b_i\})$, or some subset thereof, are defined over $\mathbb{Q}$. Using the same notation as in Section \ref{subsec_Res_padic}, if $(\{a_i\}, \{b_i\})$ are defined over $\mathbb{Q}$ with corresponding exponents $(\{p_i : 1 \leq i \leq t \}, \{q_i  : 1 \leq i \leq s \})$, then \cite{BCM}, for $A_i = T^{a_i (q-1)}$ and $T^{b_i(q-1)}$,
\begin{multline}\label{def_HypFnFF_Q}
{_{m}F_{m}} {\biggl( \begin{array}{cccc} 
A_1, & A_2, & \dotsc, & A_m \\
B_1, & B_2, & \dotsc, & B_m \end{array}
\Big| \; \lambda \biggr)}_{q}\\
=
\frac{(-1)^{t+s+1}}{q-1}
\sum_{j=0}^{q-2}
q^{-\mathfrak{s}(0)+\mathfrak{s}(-j)}
\prod_{i=1}^{t} g(T^{jp_i})
\prod_{i=1}^{s} g(T^{-jq_i})
\cdot 
T^j((-1)^{m+\delta} M^{-1} \lambda).
\end{multline}
%If $(\{a_i\}, \{b_i\})$ are defined over $\mathbb{Q}$, we can use (\ref{def_HypFnFF_Q})  as the definition of $_mF_m(\{T^{a_i (q-1)}\}; \{T^{b_i(q-1)}\} \mid \lambda \,)_q$ 
We can use (\ref{def_HypFnFF_Q}) as the definition of $_mF_m(\{T^{a_i (q-1)}\}; \{T^{b_i(q-1)}\} \mid \lambda \,)_q$ if $(\{a_i\}, \{b_i\})$ are defined over $\mathbb{Q}$,  
and then its domain is all $q$ relatively prime to $d$, the same as ${_{m}G_{m}}[\{a_i\}; \{b_i\} \mid \cdot \,]_q$. 
Letting $T= \bar{\omega}$ in (\ref{def_HypFnFF_Q}), then using the Gross-Koblitz formula, Theorem \ref{thm_GrossKoblitz}, followed by the change of variable $j \to (q-1)-j$, we get, via Theorem \ref{thm_GQ}, that 
$_mF_m(\{T^{a_i (q-1)}\}; \{T^{b_i(q-1)}\} \mid \lambda \,)_q = {_{m}G_{m}}[\{a_i\}; \{b_i\} \mid \lambda^{-1} \,]_q$.
So, if $(\{a_i\}, \{b_i\})$ are defined over $\mathbb{Q}$, then $_mF_m(\{T^{a_i (q-1)}\}; \{T^{b_i(q-1)}\} \mid \cdot \,)_q$, as defined by (\ref{def_HypFnFF_Q}), and ${_{m}G_{m}}[\{a_i\}; \{b_i\} \mid \cdot \,]_q$ are equivalent, up to inversion of the argument $\lambda$. \textbf{Therefore, Theorem \ref{thm_G_Main}; Corollary \ref{cor_G_n=2}; equation (\ref{cor_G_n=2_phi}); and, Corollaries \ref{cor_phi2}-\ref{cor_G_n=2_phi_m=4} all still hold with ${_{m}G_{m}}[\{a_i\}; \{b_i\} \mid \lambda \,]_q$ replaced by $_mF_m(\{T^{a_i (q-1)}\}; \{T^{b_i(q-1)}\} \mid \lambda^{-1} \,)_q$, as defined by (\ref{def_HypFnFF_Q}).}

The case where some, but not necessarily all, of the $(\{a_i\}, \{b_i\})$ parameters are defined over $\mathbb{Q}$ was considered by Doran et al. in \cite{DKSSVW}. Let $S_a$ and $S_b$ be submultisets of $\{a_i\}$ and $\{b_i\}$, respectively, such that
\begin{equation*}
\frac{\prod_{a \in S_a} x- e^{2 \pi \I a}}{\prod_{b \in S_b} x- e^{2 \pi \I b}} = \frac{\prod_{i=1}^{t} x^{p_i} - 1}{\prod_{i=1}^{s} x^{q_i} - 1},
\end{equation*}
for some positive integers $p_1, p_2, \dots, p_t$ and $q_1, q_2, \dots, q_s$, with $t$ and $s$ minimal. We say $(S_a, S_b)$ are defined over $\mathbb{Q}$ and we use the the same notation 
$D(x), \delta, \mathfrak{s}(c)$ and $M$, as defined in Section \ref{subsec_Res_padic}.
Let $S_a^{'}=\{a_i\} \setminus S_a$ and $S_b^{'}=\{b_i\} \setminus S_b$.
Then \cite{DKSSVW}, for $A_i = T^{a_i (q-1)}$ and $T^{b_i(q-1)}$,
\begin{multline}\label{def_HypFnFF_split}
{_{m}F_{m}} {\biggl( \begin{array}{cccc} 
A_1, & A_2, & \dotsc, & A_m \\
B_1, & B_2, & \dotsc, & B_m \end{array}
\Big| \; \lambda \biggr)}_{q}\\
\shoveleft \qquad \qquad
=
\frac{(-1)^{t+s+1}}{q-1}
\sum_{j=0}^{q-2}
q^{-\mathfrak{s}(0)+\mathfrak{s}(-j)}
\prod_{i=1}^{t} g(T^{jp_i})
\prod_{i=1}^{s} g(T^{-jq_i})\\
\times
\prod_{a \in S_a^{'}} \frac{g(T^{j+a(q-1)})}{g(T^{a(q-1)})}
\prod_{b \in S_b^{'}} \frac{g(T^{-j-b(q-1)})}{g(T^{-b(q-1)})}
\cdot 
T^j((-1)^{m+\delta} M^{-1} \lambda).
\end{multline}
So, we can use (\ref{def_HypFnFF_split}) as the definition of $_mF_m(\{T^{a_i (q-1)}\}; \{T^{b_i(q-1)}\} \mid \lambda \,)_q$ if $(S_a, S_b)$ are defined over $\mathbb{Q}$,
 and then its domain is all $q$, such that $q$ is relatively prime to $d_1$, where $d_1$ is the least common denominator of of the elements in $S_a \cup S_b$, \textbf{and}, $q \equiv 1 \imod {d_2}$, where $d_2$ is the least common denominator of of the elements in $S_a^{'} \cup S_b^{'}$.
If $(\{a_i\}, \{b_i\})$ are defined over $\mathbb{Q}$, then, taking $S_a = \{a_i\}$ and $S_b=\{b_i\}$, we recover (\ref{def_HypFnFF_Q}).

If $(S_a, S_b)$ are defined over $\mathbb{Q}$, it is straightforward to show, using the same approach as in the proof of Theorem \ref{thm_GQ}, that 
\begin{multline}\label{def_G_split}
{_{m}G_{m}}
\biggl[ \begin{array}{cccc} a_1, & a_2, & \dotsc, & a_m \\[2pt]
 b_1, & b_2, & \dotsc, & b_m \end{array}
\Big| \; \lambda \; \biggr]_q
= 
\frac{-1}{q-1}  \sum_{j=0}^{q-2} 
(-1)^{j(m+\delta)}\;
q^{-\mathfrak{s}(0)+\mathfrak{s}(j)}\;
\bar{\omega}^j(M \cdot \lambda)\\
\times
\prod_{k=0}^{r-1} 
\prod_{i=1}^{t}  
\biggfp{\langle \tfrac{-j p_i}{q-1} p^k \rangle}
(-p)^{-\lfloor{ \tfrac{-j p_i}{q-1} p^k} \rfloor}
\prod_{i=1}^{s}  
\biggfp{\langle \tfrac{j q_i}{q-1} p^k \rangle}
(-p)^{-\lfloor{ \tfrac{j q_i}{q-1} p^k} \rfloor}\\
\times
\prod_{a \in S_a^{'}}
\frac{\biggfp{\langle (a -\frac{j}{q-1} )p^k \rangle}}{\biggfp{\langle a p^k \rangle}}
(-p)^{-\lfloor{\langle a p^k \rangle -\frac{j p^k}{q-1}}\rfloor}
\prod_{b \in S_b^{'}}
\frac{\biggfp{\langle (-b +\frac{j}{q-1}) p^k \rangle}}{\biggfp{\langle -b p^k\rangle}}
(-p)^{-\lfloor{\langle -b_i p^k\rangle +\frac{j p^k}{q-1}}\rfloor}.
\end{multline}
Again, letting $T= \bar{\omega}$ in (\ref{def_HypFnFF_split}), then using the Gross-Koblitz formula, Theorem \ref{thm_GrossKoblitz}, followed by the change of variable $j \to (q-1)-j$, we get, via (\ref{def_G_split}), that 
$_mF_m(\{T^{a_i (q-1)}\}; \{T^{b_i(q-1)}\} \mid \lambda \,)_q = {_{m}G_{m}}[\{a_i\}; \{b_i\} \mid \lambda^{-1} \,]_q$.
However, in this case, if the full $(\{a_i\}, \{b_i\})$ are not defined over $\mathbb{Q}$, then the domain of $_mF_m(\{T^{a_i (q-1)}\}; \{T^{b_i(q-1)}\} \mid \cdot \,)_q$ is smaller than that of 
${_{m}G_{m}}[\{a_i\}; \{b_i\} \mid \cdot \,]_q$.

This is the reason we use $G[\cdots]_q$. While $G[\cdots]_q$ will coincide with $F(\cdots)_q$ (up to inversion of the argument $\lambda$) on any domain on which $F(\cdots)_q$ is defined, $G[\cdots]_q$ will always be defined on the largest domain possible, i.e., when $q$ is relatively prime to $d$, where $d$ is the least common denominator of the elements in $\{a_i\} \cup \{b_i\}$.

%%%%%%%%%%%%%%%%%%%%%%%%%%%%%%%%%%%%%%%%%%%%%%%%%%
%%%%%%%%%%%%%%%%%%%%%%%%%%%%%%%%%%%%%%%%%%%%%%%%%%
%%%%%%%%%%%%%%%%%%%%%%%%%%%%%%%%%%%%%%%%%%%%%%%%%%
%%%%%%%%%%%%%%%%%%%%%%%%%%%%%%%%%%%%%%%%%%%%%%%%%%
%%%%%%%%%%%%%%%%%%%%%%%%%%%%%%%%%%%%%%%%%%%%%%%%%%
%%%%%%%%%%%%%%%%%%%%%%%%%%%%%%%%%%%%%%%%%%%%%%%%%%

%\section*{Data availability}
%All data generated or analyzed during this study are included in this article.

\end{document}